\documentclass[11pt]{amsart}
\usepackage{amsmath,amssymb,xypic,array}
\usepackage[T1]{fontenc}
\usepackage{amsfonts}
\usepackage{amsmath}
\usepackage{amssymb}
\usepackage{amsthm}
\usepackage[cp850]{inputenc}
\usepackage{hyperref}
\usepackage{graphicx}

\usepackage{tikz}
\usepackage{longtable}
\usepackage{geometry}
\usepackage{textcomp}
\usetikzlibrary{trees}

\providecommand{\U}[1]{\protect\rule{.1in}{.1in}}
\providecommand{\U}[1]{\protect\rule{.1in}{.1in}}
\providecommand{\U}[1]{\protect\rule{.1in}{.1in}}
\providecommand{\U}[1]{\protect\rule{.1in}{.1in}}
\providecommand{\U}[1]{\protect\rule{.1in}{.1in}}
\input{xy}
\xyoption{all}
\theoremstyle{theorem}

\newtheorem{Theorem}{Theorem}[section]

\newtheorem{Lemma}[Theorem]{Lemma}
\newtheorem{Proposition}[Theorem]{Proposition}
\newtheorem{Problem}[Theorem]{Problem}

\theoremstyle{definition}

\newtheorem{Definition}[Theorem]{Definition}
\newtheorem{Remark}[Theorem]{Remark}

\newtheorem{Example}[Theorem]{Example}
\newtheorem{Notation}[Theorem]{Notation}
\newtheorem{Claim}[Theorem]{Claim}

\newtheorem{say}[Theorem]{}

\numberwithin{equation}{section}

\newcommand{\arXiv}[1]{\href{http://arxiv.org/abs/#1}{arXiv:#1}}
\def\bibaut#1{{\sc #1}}

\DeclareMathOperator{\NE}{\overline{NE}}
\DeclareMathOperator{\Eff}{Eff}
\DeclareMathOperator{\Nef}{Nef}

\DeclareMathOperator{\Mov}{Mov}
\DeclareMathOperator{\Cone}{Cone}

\DeclareMathOperator{\rank}{rank}

\DeclareMathOperator{\mult}{mult}

\DeclareMathOperator{\Exc}{Exc}

\DeclareMathOperator{\Sing}{Sing}

\DeclareMathOperator{\Pic}{Pic}
\DeclareMathOperator{\Jo}{Join}

\newcommand{\QED}{\ifhmode\unskip\nobreak\fi\quad {\rm Q.E.D.}} 

\newcommand\map{\dasharrow}

\renewcommand{\P}{\mathbb{P}}
\newcommand{\Q}{\mathbb{Q}}
\newcommand{\R}{\mathbb{R}}

\newcommand{\cC}{\mathcal{C}}

\renewcommand{\sec}{\mathbb{S}ec}

\newcommand{\ihat}{\mathbf {\hat \imath}}

\begin{document}
\title{Explicit log Fano structures on blow-ups of projective spaces}

\author[Carolina Araujo]{Carolina Araujo}
\address{\sc Carolina Araujo\\
IMPA\\
Estrada Dona Castorina 110\\
22460-320 Rio de Janeiro\\ Brazil}
\email{caraujo@impa.br}

\author[Alex Massarenti]{Alex Massarenti}
\address{\sc Alex Massarenti\\
IMPA\\
Estrada Dona Castorina 110\\
22460-320 Rio de Janeiro\\ Brazil}
\email{massaren@impa.br}

\date{\today}

\maketitle

\begin{abstract}
In this paper we determine which blow-ups $X$ of $\mathbb{P}^n$ at general points are log Fano, that is, when 
there exists an effective $\mathbb{Q}$-divisor $\Delta$ such that $-(K_X+\Delta)$ is ample and the pair $(X,\Delta)$ is klt.
For these blow-ups, we produce explicit boundary divisors $\Delta$ making $X$ log Fano. 
\end{abstract}

\setcounter{tocdepth}{1}
\tableofcontents


\section{Introduction}

In this paper we investigate special properties of blow-ups of complex projective spaces at general points. 
These varieties appear frequently in algebraic geometry, for example in moduli problem.
In general, the more points are blown-up, the more complicated  the resulting variety is.
For positive integers $n$ and $k$, denote by $X^n_k$ a blow-up of $\P^n$ at $k$ points in general position.
The $2$-dimensional case has been understood classically. 
For $k\leq 8$, $S=X^2_k$ is a del Pezzo surface (i.e. $-K_S$ is ample), and its geometry can be completely
described in terms of some finite data. 
For $k\geq 9$, the situation changes drastically. The anti-canonical class of $S$ is no longer big, 
and $S$ contains infinitely many $(-1)$-curves. 

For $n\geq 3$, $X^n_1$ is a Fano manifold  (i.e., $-K_{X^n_1}$ is ample), but as soon as $k\geq 2$, 
$X^n_k$ is no longer Fano. However, for small values of $k$, the blow-up $X^n_k$ behaves like a Fano manifold. 
A more appropriate notion here is that of a \textit{log Fano variety}. 

\begin{Definition}\label{def1}
Let $X$ be a normal projective $\mathbb{Q}$-factorial variety. 
We say that $X$ is \textit{log Fano} if there exists an effective $\mathbb{Q}$-divisor $\Delta$ such that $-(K_X+\Delta)$ is ample, 
and the pair $(X,\Delta)$ is klt. (See for instance \cite[Definition 3.5]{Ko} for the notion of klt singularities). 
\end{Definition}

Log Fano varieties play an important role in the classification of algebraic varieties. 
It was proved in \cite[Corollary 1.3.2]{BCHM} that they are special instances of 
\textit{Mori dream spaces}. 
We refer to Section \ref{section:MDS} and references therein for the definition and 
special properties of Mori dream spaces. Here we just vaguely remark that 
a Mori dream space $X$ behaves very well with respect to the minimal model program.
Moreover, the birational geometry of $X$ can be encoded 
in some finite data, namely its cone of effective divisors $\Eff(X)$ together with 
a chamber decomposition on it. 

In this paper we address the following problems:

\begin{center}
\textit{
For which values of $n$ and $k$ is $X^n_k$ a log Fano variety? \\
In those cases, can we find an explicit $\mathbb{Q}$-divisor $\Delta$ making $X^n_k$ log Fano?}
\end{center} 

In the second question, the expected properties of $\Delta$ depend on the context.
In some cases, one would like $\Delta$ to be irreducible.
On the other hand, when $X^n_k$ appears as a compactification of some moduli space, it is often desirable that 
$\Delta$ is supported on the boundary divisor. 

\begin{Example}\label{p3}
Let us consider the case $n=3$.
By Proposition \ref{mc1} below, if $k\leq 8$, then the Mori cone of $X=X^3_k$ is generated by classes of lines $R_i$
in the exceptional divisors, $1\leq i\leq k$, and  strict transforms $L_{i,j}$ of the lines through two of the 
blown-up points. 
Using this result, one checks easily that 
$-K_X$ is nef for $k\leq 8$. Moreover, by computing the top intersection $(-K_X)^3$, one concludes that
$-K_X$ is big if and only if $k\leq 7$.
Projective manifolds with nef and big anti-canonical class are called 
\textit{weak Fano}. The fact that $X_k^3$ is weak Fano if and only if $k\leq 7$ has been proven, with slightly different techniques,
in \cite[Proposition 2.9]{BL}. 
By Lemma~\ref{lemma:wfano=>logfano}, weak Fano manifolds are log Fano. 
On the other hand, log Fano varieties have big anti-canonical class. 
So we conclude that $X$ is log Fano if and only if $k\leq 7$. 

When $k\leq 4$, $X_k^3$ is a toric variety and one can take $\Delta$ to be a suitable combination of toric invariant divisors.
Alternatively, we may choose $\Delta=\epsilon D$ to be irreducible. 
We describe such irreducible  $\Delta$ when $k=4$.  
We may assume that the blown-up points are the fundamental points of $\mathbb{P}^3$.
Let $D\subset X^3_4$ be  the strict transform of the Cayley nodal cubic surface
$$
S = \{x_0x_1x_2+x_0x_1x_3+x_0x_2x_3+x_1x_2x_3 = 0\}\subset\mathbb{P}^3.
$$
The surface $S$ has ordinary double points at the fundamental points of $\mathbb{P}^3$, and is smooth elsewhere. 
Thus $D$ is smooth and $D \sim 3H-2(E_1+...+E_4)$.
One computes that 
$$
(-K_{X_4^3}-\epsilon D)\cdot R_i = 2-2\epsilon \ \text{ and } \  (-K_{X_4^3}-\epsilon D)\cdot L_{i,j} =  4\epsilon.
$$
Thus $(X, \epsilon D)$ is klt and  $-(K_{X_4^3}+\epsilon D)$ is ample for any $0 < \epsilon < 1$.

When $k = 5$, let $H_{i,j,k}\subset \P^3$ be the plane spanned by three of the blown-up points $p_i$, $p_j$ and $p_k$, and
take $D$ to be the strict transform of $\sum_{i,j,k}H_{i,j,k}$. 
Then $D \sim 10H-6(E_1+...+E_5)$, and one computes 
$$
-(K_{X_5^3}-\epsilon D)\cdot R_i = 2-6\epsilon  \ \text{ and } \  -(K_{X_5^3}-\epsilon D)\cdot L_{i,j} = 2\epsilon.
$$ 
So $-(K_{X_5^3}+\epsilon D)$ is ample for any $0 < \epsilon < \frac{1}{3}$. Furthermore, we can take $\epsilon>0$ sufficiently small so that the pair 
$(X^3_5,\epsilon D)$ is klt.

When $k=6$, let $Q_i\subset \P^3$ be the unique irreducible quadric cone through the $6$ blown-up points, with vertex at one of the points $p_i$.
Let $D$ be the strict transform of $Q_1+Q_2+Q_3+H_{4,5,6}$ in $X^3_6$. 
(As before, $H_{4,5,6}\subset \P^3$ denotes the plane spanned by $p_4$, $p_5$ and $p_6$.)
Then $D \sim 7H-4(E_1+...+E_6)$, and one computes 
$$
-(K_{X_6^3}+\epsilon D)\cdot R_i = 2-4\epsilon  \ \text{ and } \   -(K_{X_6^3}+\epsilon D)\cdot R_i = \epsilon.
$$
So $-(K_{X_6^3}+\epsilon D)$ is ample for any $0 < \epsilon < \frac{1}{2}$.
Furthermore we can take $\epsilon>0$ sufficiently small so that the pair 
$(X^3_6,\epsilon D)$ is klt.

Suppose now that $k=7$.
For each triple $1\leq i,j,k \leq 7$, consider the linear system of cubics defining the standard Cremona transformation of $\P^3$ centered at the 
four points  $p_h$, $h\neq i,j,k$. 
There exists an unique irreducible cubic surface $S_{i,j,k}\subset \P^3$ in this linear system passing through $p_i$, $p_j$ and $p_k$.
It is smooth at $p_i$, $p_j$ and $p_k$, and has a double point at $p_h$ for $h\neq i,j,k$. Its strict transform in 
$X^3_7$ is a rigid surface. 
Let $D$ be the strict transform of $\sum_{i,j,k}S_{i,j,k}$ in $X^3_7$. 
Then $D \sim 105H-55(E_1+...+E_7)$, and one computes 
$$
-(K_{X_7^3}+\epsilon D)\cdot R_i = 2-55\epsilon \ \text{ and } \   (K_{X_7^3}+\epsilon D)\cdot L_{i,j} = 5\epsilon.
$$
So $-(K_{X_5^3}+\epsilon D)$ is ample for any $0 < \epsilon < \frac{2}{55}$. 
Furthermore we can take $\epsilon>0$ sufficiently small so that  the pair 
$(X^3_7,\epsilon D)$ is klt.
\end{Example}

For $n\geq 4$, we can approach the first question using Mori dream spaces.
The main results of \cite{Muk01} and \cite{CT} put together show that $X^n_k$
is a Mori dream space if and only if one of the following holds:
\begin{itemize}
	\item[-] $n=4$ and $k\leq 8$.
	\item[-] $n>4$ and $k\leq n+3$.
\end{itemize}
Using Mukai's description of the geometry of the
boundary cases $X^4_8$ and $X^n_{n+3}$, see Section \ref{effectivecone}, we answer the first question. 

\begin{Theorem}\label{Thm:logFanos}
Let $X^n_k$ be a blow-up of $\P^n$ at $k$ points in general position, with $n\geq 2$ and $k\geq 0$.
Then  $X^n_k$ is log Fano if and only if one of the following holds:
\begin{itemize}
	\item[-] $n=2$ and $k\leq 8$,
	\item[-] $n=3$ and $k\leq 7$,
	\item[-] $n=4$ and $k\leq 8$,
	\item[-] $n>4$ and $k\leq n+3$.
\end{itemize}
\end{Theorem} 

The proof of Theorem \ref{Thm:logFanos}, which may already have been known to experts,  
does not give any hint on which $\mathbb{Q}$-divisor $\Delta$ makes $X^n_k$ log Fano. 
So we proceed to find such explicit $\mathbb{Q}$-divisor $\Delta$.
The first step is to determine the Mori cone of $X^n_k$.

\begin{Proposition}\label{mc1} \label{mc2}
Let $X^n_k$ be the blow-up of $\P^n$ at points in general position $p_1,\dots, p_k$, $n\geq 2$.
Denote by $R_i$ the class of a line in the exceptional divisors over $p_i$, and by
$L_{i,j}$ the class of the  strict transforms of the line through two distinct points $p_i$ and $p_j$. 
Suppose that either of the following holds:
\begin{itemize}
	\item[-] $k\leq 2n$.
	\item[-] $n=3$ and $k\leq 8$.
\end{itemize}
Then the Mori cone $\NE(X^n_k)$ is generated by the  $R_i$'s and $L_{i,j}$'s.
\end{Proposition} 

Using Proposition \ref{mc1}, it is not hard to find a $\mathbb{Q}$-divisor $\Delta$ such that $-(K_{X^n_k}+\Delta)$ is ample.
We often choose $\Delta$ as linear combinations of extremal divisors in $X^n_k$.
The hard part is to show that for such divisors  $(X^n_k,\Delta)$ is klt.
We do so by providing explicit log resolutions for these pairs, and computing discrepancies. 
Explicit $\mathbb{Q}$-divisors $\Delta$ making $X^n_k$ log Fano
are given in Theorems~\ref{lfn+1}, \ref{n+2}, \ref{mainodd} and \ref{maineven}.
In particular, they provide a new proof that these varieties are Mori dream spaces.

Some blow-ups of  projective spaces at points (and, more generally, linear spaces) appear as 
moduli spaces $\overline{M}_{g,A[n]}$ of weighted pointed stable curves.
These spaces were introduced and investigated by Hassett in \cite{Ha}. 
In \cite[Problem 7.1]{Ha}, Hassett asks whether there is an effective $\mathbb{Q}$-divisor $\Delta$ 
on $\overline{M}_{g,A[n]}$, supported on the boundary, such that $(\overline{M}_{g,A[n]}, \Delta)$
is log canonical, and $ K_{\overline{M}_{g,A[n]}}+ \Delta$ is ample. 
We end the paper by addressing this question.

The paper is organized as follows. 
In Section~\ref{section:MDS}, we recall  the definition and some
special properties of Mori dream spaces. 
In Section~\ref{section:MCD_X} we review the description from \cite{CT} of the  cone of effective divisors 
of $X^n_k$, and make explicit the description of its Mori chamber decomposition
proposed in \cite{MukaiADE}.
We end this section by proving Theorem~\ref{Thm:logFanos}.
In Section~\ref{section:k<n+3}, we exhibit an integral divisor $D\subset X_k^n$ and rational number $\epsilon>0$ such that 
$\Delta = \epsilon D$ makes  $X^n_k$ log Fano for $k\leq n+2$.
This task is relatively easy, and serves as warm up for the next case $n=k+3$, treated 
in Section~\ref{section:k=n+3}.
For $X^n_{n+3}$, we construct  $D$ from  joins of suitable linear spaces and higher secant varieties 
of the unique rational normal curve through the blown-up points.
In order to construct an explicit log resolution for the resulting pair  $(X^n_{n+3},\Delta)$,
we  need a good understanding of the intersections of such joins. 
Subsection~\ref{subsec:sec} is devoted to this.
The description of  $D$ is given separately when $n$ is odd (Subsection~\ref{subset:odd}) 
and even (Subsection~\ref{subset:even}).
Finally, in Section~\ref{msc}, we address a question of Hassett about some 
moduli spaces $\overline{M}_{g,A[n]}$ of weighted pointed stable curves.

\

\noindent {\bf {Acknowledgments.}} 
This work was done while the second named author was a Post-Doctorate at IMPA, funded by CAPES-Brazil. The first named author was partially supported by CNPq and Faperj Research Fellowships. We would like to thank Cinzia Casagrande, Ana-Maria Castravet and Massimiliano Mella for 
inspiring conversations that helped giving shape to this paper.


\section{Mori Dream Spaces} \label{section:MDS}

Let $X$ be a normal projective variety. 
We denote by $N^1(X)$ the real vector space of $\R$-Cartier divisors modulo numerical equivalence. 
The \emph{nef cone} of $X$ is the (closed) convex cone $\Nef(X)\subset N^1(X)$ generated by classes of 
nef divisors. 
The \emph{movable cone} of $X$ is the convex cone $\Mov(X)\subset N^1(X)$ generated by classes of 
\emph{movable divisors}. These are Cartier divisors whose stable base locus has codimension at least two in $X$.
The \emph{effective cone} of $X$ is the convex cone $\Eff(X)\subset N^1(X)$ generated by classes of 
\emph{effective divisors}.
We have inclusions:
$$
\Nef(X)\ \subset \ \overline{\Mov(X)}\ \subset \ \overline{\Eff(X)}.
$$

We say that a birational map  $f: X \dasharrow X'$ into a normal projective variety $X'$  is a \emph{birational contraction} if its
inverse does not contract any divisor. 
We say that it is a \emph{small $\Q$-factorial modification} 
if $X'$ is $\Q$-factorial  and $f$ is an isomorphism in codimension one.
If  $f: X \dasharrow X'$ is a small $\Q$-factorial modification, then 
the natural pullback map $f^*:N^1(X')\to N^1(X)$ sends $\Mov(X')$ and $\Eff(X')$
isomorphically onto $\Mov(X)$ and $\Eff(X)$, respectively.
In particular, we have $f^*(\Nef(X'))\subset \overline{\Mov(X)}$.

\begin{Definition}\label{def:MDS} 
A normal projective $\Q$-factorial variety $X$ is called a \emph{Mori dream space}
if the following conditions hold:
\begin{enumerate}
\item[-] $\Pic{(X)}$ is finitely generated,
\item[-] $\Nef{(X)}$ is generated by the classes of finitely many semi-ample divisors,
\item[-] there is a finite collection of small $\Q$-factorial modifications
 $f_i: X \dasharrow X_i$, such that each $X_i$ satisfies the second condition above, and 
 $$
 \Mov{(X)} \ = \ \bigcup_i \  f_i^*(\Nef{(X_i)}).
 $$
\end{enumerate}
\end{Definition}

The collection of all faces of all cones $f_i^*(\Nef{(X_i)})$'s above forms a fan supported on $\Mov(X)$.
If two maximal cones of this fan, say $f_i^*(\Nef{(X_i)})$ and $f_j^*(\Nef{(X_j)})$, meet along a facet,
then there exists a commutative diagram: 
  \[
  \begin{tikzpicture}[xscale=1.5,yscale=-1.2]
    \node (A0_0) at (0, 0) {$X_i$};
    \node (A0_2) at (2, 0) {$X_{j}$};
    \node (A1_1) at (1, 1) {$Y$};
    \path (A0_0) edge [->,swap]node [auto] {$\scriptstyle{h_i}$} (A1_1);
    \path (A0_2) edge [->]node [auto] {$\scriptstyle{h_j}$} (A1_1);
    \path (A0_0) edge [->,dashed]node [auto] {$\scriptstyle{\varphi}$}  (A0_2);
  \end{tikzpicture}
  \]
where $Y$ is a normal projective variety, and $h_i$ and $h_j$ are small birational morphisms. 
The fan structure on $\Mov(X)$ can be extended to a fan supported on $\Eff(X)$ as follows. 

\begin{Definition}\label{MCD}
Let $X$ be a Mori dream space.
We describe a fan structure on the effective cone $\Eff(X)$, called the \emph{Mori chamber decomposition}.  
We refer to \cite[Proposition 1.11(2)]{HK} and \cite[Section 2.2]{okawa_MCD} for details.
There are finitely many birational contractions from $X$ to Mori dream spaces, denoted by $g_i:X\map Y_i$.
The set $\Exc(g_i)$ of exceptional prime divisors of $g_i$ has cardinality $\rho(X/Y_i)=\rho(X)-\rho(Y_i)$.
The maximal cones $\cC_i$ of the Mori chamber decomposition of $\Eff(X)$ are of the form:
$$
\cC_i \ = \Cone \ \Big( \ g_i^*\big(\Nef(Y_i)\big)\ , \  \Exc(g_i) \ \Big).
$$
We call $\cC_i$ or its interior $\cC_i^{^\circ}$ a \emph{maximal chamber} of $\Eff(X)$.
\end{Definition}

By \cite[Corollary 1.3.2]{BCHM}, a log Fano variety is a Mori dream space.
The converse does not hold in general, and there are several criteria for a Mori dream space to be log Fano \cite{GOST}. We will use the following. 

\begin{Proposition}\label{lfsmall}
Let $X$ be a log Fano variety. Then any small $\Q$-factorial modification of $X$ is also log Fano. 
\end{Proposition}

Proposition~\ref{lfsmall} follows from the properties of Mori dream spaces and Lemma~\ref{GOST} below.
In what follows, a normal projective variety $X$ is said to be of Fano type if 
there exists an effective $\mathbb{Q}$-divisor $D$ on $X$ such that $-(K_X+D)$ is $\Q$-Cartier and ample, 
and the pair $(X,D)$ is klt. 
This is weaker than our current notion of log Fano because it does not require that $X$ be $\Q$-factorial.

\begin{Lemma}[{\cite[Lemma 3.1]{GOST}}] \label{GOST}
Let $h:X\to Y$ be a small birational morphism between normal projective varieties. 
Then $X$ is of Fano type if and only if so is $Y$.
\end{Lemma}

\begin{Lemma}\label{lemma:wfano=>logfano}
Let $X$ be a normal $\Q$-factorial  projective variety with at worst klt singularities.
Suppose that $-K_X$ is nef and big.
Then $X$ is log Fano. 
\end{Lemma} 

\begin{proof}
Since $-K_X$ is big, by \cite[Corollary 2.2.6]{La}, there exist an ample divisor $A$, an effective divisor $D$, and a positive integer $m$ such that 
$-mK_X\equiv A+D$. For $h>m$, we can write $-hK_X\equiv -(h-m)K_X+A+D$. 
The divisor $D^{'} = -(h-m)K_X+A$ is a sum of a nef and an ample divisor, and so it is ample. 
Setting $\epsilon = \frac{1}{h}$, we get that
$-(K_X+\epsilon D)\equiv \epsilon D^{'}$
is ample. Since $X$ has at worst klt singularities, by taking $h$ large enough, we get that the pair $(X, \epsilon D)$ is klt.
\end{proof}

\begin{Remark}\label{cinzia}
In the next section, we will use Proposition~\ref{lfsmall}  and Lemma~\ref{lemma:wfano=>logfano}  to 
prove that certain blow-ups $X$ of $\P^n$ at general points are Mori dream spaces.
To do so, we will use the fact that $X$ admits a small $\Q$-factorial modification $X'$ which is smooth and has $-K_{X'}$ nef and big. 
Notice that smoothness of $X'$ is essential.
In fact, there are examples of Mori dream spaces $X$ which are not log Fano, but admit 
a (very singular) small $\Q$-factorial modification $X'$ with $-K_{X'}$ ample, see for instance \cite[Example 5.1]{CG}.
\end{Remark}


\section{Cones of curves and divisors on blow-ups of $\P^n$ at general points} \label{section:MCD_X}

Let $p_1,...,p_k\in\mathbb{P}^n$ be general points, and let $X^n_k$ be the blow-up of $\mathbb{P}^n$ at $p_1,...,p_k$. 
In this section we describe several cones of curves and divisors on $X^n_k$. 

\begin{Notation}\label{notation:generators}
We denote by $H\in N^1(X^n_k)$ the class of the pullback of the hyperplane section of $\mathbb{P}^n$.
By abuse of notation, we denote by  $E_i$ both the exceptional divisor over $p_i$ and its class in $N^1(X^n_k)$. 
Then
$\{H, E_1, \dots, E_k\}$ is a basis of $N^1(X^n_k)$, and we have
$$
-K_{X_{k}^n} = (n+1)H-(n-1)E_{1}-...-(n-1)E_{k}.
$$
We denote by $L\in N_1(X^n_k)$ the class of the strict transform of a general line on $\mathbb{P}^n$.
For each $i\in\{1, \dots, k\}$, we denote by $R_i\in N_1(X^n_k)$ the class of a line on $E_i\cong \P^{n-1}$, and by 
$L_i\in N_1(X^n_k)$ the class of the strict transform of a general line on $\mathbb{P}^n$ passing through $p_i$.
For $i\neq j$, we denote by $L_{i,j}$ the class of the strict transform of the line on $\mathbb{P}^n$ joining $p_i$ and $p_j$.
Then $\{L, R_1, \dots, R_k\}$ is a basis of $N_1(X^n_k)$, and we have
\begin{equation}\label{eq:RandL}
L \equiv L_{i,j}+R_i+R_j \ \text{ and } \ L_i\equiv L-R_i\equiv L_{i,j}+R_j.  
\end{equation}
\end{Notation}


\subsection{The Mori cone of $X_k^n$}\label{moricone}
In this section we prove Proposition~\ref{mc1}.

\begin{Lemma}\label{lemma2}
Let $p_1,...,p_8\in\mathbb{P}^3$ be general points, and $C\subset\mathbb{P}^{3}$ an irreducible curve of degree $d$ having multiplicity 
$m_i = \mult_{p_i}(C)$ at $p_i$, $1\leq i\leq 8$.
Then   $m_1+...+m_8 \leq 2d$.
\end{Lemma}
\begin{proof}
If $C$ is degenerate, then $m_i\neq 0$ for at most three points $p_i$, and the conclusion follows easily from B\'ezout's Theorem. 
So from now on we assume that $C$ is non degenerate. 
Let $\Lambda$ be the pencil of irreducible quadric surfaces passing through $p_1,\dots ,p_8$.
Suppose that $m_1+...+m_8 > 2d$.
It follows from B\'ezout's Theorem  that $C$ is contained in every member of $\Lambda$.  
In particular,  $C$ is a non degenerate irreducible curve contained in the intersection of two irreducible quadric surfaces. So $d\in \{3,4\}$.
Suppose that $d=3$. Then $C$ must be a twisted cubic through at most $6$ of the $p_i$'s, and thus  $m_1+...+m_8 \leq 6=2d$,
contradicting our assumptions. 
We conclude that $d=4$, $m_i\geq 1$ for every $i$, and $m_j\geq 2$ for some $j$. 
If follows from B\'ezout's Theorem that $m_j= 2$, and $m_i = 1$ for $i\neq j$.  
Consider the projection from $p_j$
$$
\pi_{p_j}:C\dasharrow\mathbb{P}^{2}.
$$
The image $\overline{\pi_{p_1}(C)}$ is a conic though the seven general points $\pi_{p_j}(p_i)$, $i \neq j$,
which is impossible. This shows that $m_1+...+m_8 \leq 2d$.
\end{proof}

\begin{proof}[Proof of Proposition~\ref{mc1}]
Let $X^n_k$ be the blow-up of $\P^n$, $n\geq 2$, at points in general position $p_1,\dots, p_k$.
We follow Notation~\ref{notation:generators}.
Let $\widetilde{C}\subset X^n_k$ be an irreducible curve not  contained in any exceptional divisor $E_i$,
and denote by $C$ the image of $\widetilde{C}$ in $\P^n$.
It is an irreducible curve of degree $d>0$ and multiplicity $m_i = \mult_{p_i}C\geq 0$ at $p_i$, 
 $\widetilde{C}$ is the strict transform of $C$, and 
\begin{equation}\label{expcurv}
\widetilde{C}\equiv dL-m_1R_1-...-m_kR_k.
\end{equation}
We must show that the class of $\widetilde{C}$ in $N_1(X^n_k)$ lies in the cone generated by the $R_i$'s and $L_{i,j}$'s.
We may assume that $m_1\leq m_2 \leq \cdots \leq m_k$.

Suppose first that $k\leq 2n$. 

First let us assume that $k$ is even. We  write
\begin{eqnarray}\label{even}
\begin{array}{cl}
\widetilde{C}\equiv & dL-m_1(R_1+R_2)-(m_2-m_1)R_2-m_3(R_3+R_4)-(m_4-m_3)R_4-\\ 
 & ...-m_{k-1}(R_{k-1}+R_k)-(m_k-m_{k-1})R_k.
\end{array} 
\end{eqnarray}
Note that  $m_1+(m_2-m_1)+m_3+(m_4-m_3)+...+m_{k-1}+(m_k-m_{k-1}) = m_2+m_4+...+m_k$. 
We claim that $m_2+m_4+...+m_k\leq d$ .
Indeed, since $k\leq 2n$, the set $\{p_2,p_4,...,p_{k}\}$ has cardinality at most $n$.  Consider the linear space 
$P = \left\langle p_2,p_4,...,p_k\right\rangle\subsetneqq\mathbb{P}^n$. 
If $m_2+m_4+...+m_k > d$, then $C\subset P$ by B\'ezout's Theorem .
Since the $p_i$'s are general, $p_1,p_3,...,p_{k-1}\not\in P$, and so 
$m_1 = m_3 = ... = m_{k-1} = 0$. 
But this implies that $m_i=0$ for $i\leq k-1$ and $m_k>d$, which is impossible. 
This proves the claim. 
So we can rewrite \eqref{even} as
$$
\begin{array}{cl}
\widetilde{C}\equiv & m_1L_{1,2}+(m_2-m_1)L_2+m_3L_{3,4}+(m_4-m_3)L_4-\\ 
 & ...+m_{k-1}L_{k-1,k}+(m_k-m_{k-1})L_k+(d-m_2-m_4-...-m_k)L.
\end{array} 
$$
It follows from \eqref{eq:RandL} that the class of $\widetilde{C}$  in $N_1(X^n_k)$ lies in the cone generated by the $R_i$'s and $L_{i,j}$'s.

Now suppose that $k$ is odd, and write 
\begin{eqnarray}\label{odd}
\begin{array}{cl}
\widetilde{C}\equiv & dL-m_1(R_1+R_2)-(m_2-m_1)R_2-m_3(R_3+R_4)-(m_4-m_3)R_4-\\ 
 & ...-m_{k-2}(R_{k-2}+R_{k-1})-(m_{k-1}-m_{k-2})R_{k-1}-m_kR_k.
\end{array} 
\end{eqnarray}
In this case 
$m_1+(m_2-m_1)+m_3+(m_4-m_3)+...+m_{k-2}+(m_{k-1}-m_{k-2})+m_k = m_2+m_4+...+m_{k-1}+m_k$.
Like in the even case, one shows that $m_2+m_4+...+m_{k-1}+m_k\leq d$ and 
rewrite \eqref{odd} as an effective linear combination of the $R_i$'s and $L_{i,j}$'s.

From now on we suppose that  $n=3$ and $k\leq 8$.
Then $m_i\leq d$ and  $m_1+...+m_k\leq 2d$ by Lemma~\ref{lemma2}.
If $m_{k-1}=0$, then $\widetilde{C}\equiv  m_k L_k + (d-m_k)L$. 
It follows from \eqref{eq:RandL} that the class of $\widetilde{C}$  in $N_1(X^n_k)$ lies in the cone generated by the $R_i$'s and $L_{i,j}$'s.
If $m_{k-1}\neq 0$, then rewrite \eqref{expcurv} as
$$
\widetilde{C}\equiv (L_{k-1,k}) +  d'L-m_1'R_1-...-m_k'R_k,
$$
where $d'=d-1$, $m_i'=m_i$ for $i\leq k-2$, and $m_i'=m_i-1$ for $i=k-1$ or $k$.
Note that $m_i'\leq d'$. This is clear for   $i=k-1$ or $k$. For $i\leq k-2$ it follows from the 
assumptions that $m_1\leq m_2 \leq \cdots \leq m_k\leq d$ and $m_1+...+m_k\leq 2d$.
We also have $m_1'+...+m_k'\leq 2d'$.
So we can repeat the process and conclude by induction that  the class of 
$\widetilde{C}$  in $N_1(X^n_k)$ lies in the cone generated by the $R_i$'s and $L_{i,j}$'s.
\end{proof}


\subsection{The effective cone of $X_{n+3}^n$}\label{effectivecone}

In this section we describe the effective cone of the blow-up of $\P^n$ at $n+3$ points in general position,
as well as its Mori chamber decomposition. The main references are \cite{CT}, \cite{MukaiADE} and \cite{Bauer}. See also \cite{BDP15} for a recent new proof.

\begin{say}[The effective cone of the blow-up of $\P^n$ at $n+3$ points]\label{Eff(X)}
Let $X=X^n_{n+3}$ be the blow-up of $\P^n$ at $n+3$ points $p_i$ in general position.
We follow Notation~\ref{notation:generators}.
By \cite[Theorems 1.3]{CT}, $X$ is a Mori dream space.
Next we describe the $1$-dimensional faces of $\Eff(X)$ (\cite[Theorem 1.2]{CT}).
For each subset $I\subset \{1, \cdots, n+3\}$ whose complement  has odd cardinality $|I^{^c}|=2k+1$,
consider the divisor class
$$
E_I \ := \ kH - k\sum_{i\in I}E_i -(k-1) \sum_{i\in I^{^c}}E_i.
$$
There is a unique divisor in the linear system $\big|E_I\big|$, which we also denote by $E_I$.
When $k=0$ we have $E_{\{i\}^{^c}}=E_i$
When $k\geq 1$, $E_I$ can be described as follows.
Let $\pi_I:\P^n\map \P^{2k-2}$ be the projection from the linear space $\langle p_i\rangle_{i\in I}$. 
Let $C_I\subset \P^{2k-2}$ be the image of the unique rational normal curve through all the $p_i's$.
The divisor $E_I$ is the cone with vertex $\langle p_i\rangle_{i\in I}$ over $\sec_{k-1}C_I$. 
Each $E_I$ generates a $1$-dimensional face of $\Eff(X)$, and all $1$-dimensional faces
are of this form.
\end{say} 

Let $X$ be the blow-up of $\P^n$ at $n+3$ points  in general position,
and follow the notation of Paragraph~\ref{Eff(X)} above. 
In order to describe the Mori chamber decomposition of $\Eff(X)$, we make explicit the map to the weight space proposed by Mukai in \cite{MukaiADE}.
Write $(y, x_1,\dots, x_{n+3})$ for coordinates in  $\R^{n+4}$, and $(\alpha_1, \dots, \alpha_{n+3})$
for coordinates in $\R^{n+3}$. We identify $\R^{n+4}$ with $N^1(X)$ by associating to a point
$\bar x= (y, x_1,\dots, x_{n+3})\in \R^{n+4}$  the divisor class of 
$D_{\bar x}= yH+\sum x_i E_i$. 
Note that all the $E_I$'s defined in Paragraph~\ref{Eff(X)} lie in the hyperplane 
$$
(n+1)y +\sum x_i \ = \ 1.
$$ 
Consider the projection from the origin
\begin{equation} \label{eq:phi}
\begin{aligned}
\varphi \ = \ (\varphi_1, \cdots, \varphi_{n+3}) \ & : \ \Eff(X) \ \to \ \R^{n+3},  \\
\varphi_i   \ = \  &\frac{y+x_i}{(n+1)y +\sum x_i}.
\end{aligned}
\end{equation}
We shall describe  the image of $\Eff(X)$ under $\varphi$, along with the decomposition 
induced by the Mori chamber decomposition of $\Eff(X)$. Before we do so, let us introduce some notation. 
The vertices of the hypercube $[0,1]^{n+3}\subset \R^{n+3}$ are the points of the form 
$\xi_I=\big((\xi_I)_1, \dots, (\xi_I)_{n+3}\big)$, where 
$I\subset \{1, \dots, n+3\}$, $(\xi_I)_i=1$ if $i\in I$, and  $(\xi_I)_i=0$ otherwise.
The parity of the vertex $\xi_I$ is the parity of $|I|$.
For each subset $I\subset \{1, \dots, n+3\}$, define the degree one polynomial in the $\alpha_i$'s:
\begin{equation} \label{eq:H_I}
H_I \ := \ \sum_{j\not\in I} \alpha_j + \sum_{i\in I}(1-\alpha_i). 
\end{equation} 
For any  subset $J\subset \{1, \dots, n+3\}$, we have:
\begin{equation} \label{eq:H_I(x_J)}
H_I(\xi_J) \ = \ \# (I^{^c}\cap J) + \# (J^{^c}\cap I). 
\end{equation} 
Given $J\subset \{1, \dots, n+3\}$ and $i_0\not\in J$, set $I:=J\cup \{i_0\}$.
Then
\begin{equation} \label{eq:H_I/H_J}
H_I \ = \ H_J+1-2\alpha_{i_0}. 
\end{equation} 
One computes that $\varphi(E_I) = \xi_{I^{^c}}$. 
Therefore, the image of $\Eff(X)$ under $\varphi$ is the polytope $\Delta \subset \R^{n+3}$
generated by the odd vertices of the hypercube. 
Using \eqref{eq:H_I(x_J)} above, one can easily check that  
the polytope $\Delta\subset \R^{n+3}$ is defined by the following set of inequalities: 
\begin{equation} \label{eq:Delta}
\Delta \ = \ \varphi\big(\Eff(X)\big) \ = \  \left\{ 
\begin{aligned}
& 0\leq \alpha_i \leq 1, \ & i\in \{1, \dots, n+3\} \\
&H_I\geq 1 , \ & |I| \text{ even.}
\end{aligned}
\right.
\end{equation} 
Next we describe the chamber decomposition in $\Delta$ induced by the  Mori chamber decomposition of $\Eff(X)$.
For each subset $I\subset \{1, \dots, n+3\}$, and each integer $k$ satisfying $2\leq k\leq \frac{n+3}{2}$
and $|I|\not\equiv k \mod 2$, consider the hyperplane 
$(H_I =  k)$.
Now take the complement in the interior of $\Delta$ of the hyperplane arrangement
\begin{equation} \label{eq:MCD_Delta}
  \Big(\ H_I \ = \  k \ \Big)_{\ 2\leq k\leq \frac{n+3}{2}, \ |I|\not\equiv k \mod 2.}
\end{equation} 
and consider its decomposition into connected components. Each connected component is called a \emph{chamber} of $\Delta$.

The following theorem summarizes the results of \cite{MukaiADE} and \cite{Bauer}.
The proof follows from the proof of the main theorem in \cite[Page 6]{MukaiADE} and the description
of wall crosses in \cite[Propositions 2 and 3]{MukaiADE}. 
Mukai's proof relies on interpreting $X$ as a moduli space of parabolic vector bundles on $\P^1$,
and the description of these spaces in \cite[Section 2]{Bauer}.

\begin{Theorem}\label{thm:mukai-bauer}
Let $X$ be the blow-up of $\P^n$ at $n+3$ points in general position,
and consider the projection 
$$
\varphi \ : \ \Eff(X) \ \to \ \Delta
$$
defined in \eqref{eq:phi} above. 
\begin{enumerate}
	\item[-] The chamber decomposition of $\Delta$ defined by the hyperplane arrangement \eqref{eq:MCD_Delta}
		coincides with that induced by the Mori chamber decomposition of $\Eff(X)$ via $\varphi$. 
		
	\item[-] The image of $\Mov(X)$ under $\varphi$ is given by 
$$
		\Pi \ = \ \varphi\big(\Mov(X)\big) \ = \  \left\{ 
		\begin{aligned}
		& 0\leq \alpha_i \leq 1, \ & i\in \{1, \dots, n+3\} \\
		&H_I\geq 2 , \ & |I| \text{ odd.}
		\end{aligned}
		\right.
$$
	\item[-] All small $\Q$-factorial modifications of $X$ are smooth.
		Let $\cC$ and $\cC'$ be two adjacent chambers of $\Mov(X)$, corresponding 
		to small $\Q$-factorial modifications of $X$, $f:X\map \widetilde{X}$ and $f':X\map\widetilde{X}'$, respectively.
		The images of these chambers in $\Delta$ are separated by a hyperplane of the form
		$(H_I =  k)$, with $3\leq k\leq \frac{n+3}{2}$ and $|I|\not\equiv k \mod 2$. 
		Suppose that $\varphi(\cC)\subset (H_I \leq  k)$ and $\varphi(\cC')\subset (H_I \geq  k)$.
		Then the birational map $f'\circ f^{-1}:\widetilde{X} \map \widetilde{X}'$ flips a $\P^{k-2}$ into a $\P^{n+1-k}$.
	\item[-] Let $\cC$ be a chamber of $\Mov(X)$, corresponding to small $\Q$-factorial modification $\widetilde{X}$ of $X$.
		Let $\sigma\subset \partial \cC$ be a wall such that $\sigma\subset \partial \Mov(X)$, and let 
		$f:\widetilde{X}\to Y$ be the corresponding elementary contraction.
		The image of $\sigma$ in $\Pi$ is supported on a hyperplane of one of the following forms:
		\begin{enumerate}
			\item[\textit{(a)}] ($\alpha_i=0$) or ($\alpha_i=1$). 
			\item[\textit{(b)}] ($H_I= 2$), with $|I|$ odd. 
		\end{enumerate}
		In case \textit{(a)}, $f:\widetilde{X}\to Y$ is a $\P^1$-bundle. In case \textit{(b)}, $f:\widetilde{X}\to Y$ is the blow-up of a smooth point,
		and the exceptional divisor of $f$ is the image in $\widetilde{X}$ of the divisor $E_{I^{^c}}$.
\end{enumerate}
\end{Theorem}

\begin{Remark}
In Theorem~\ref{thm:mukai-bauer}, note that 
$$
\Delta\cap(H_I\geq 2)_{_{|I| \text{ odd}}} \ = \  [0,1]^{n+3}\cap(H_I\geq 2)_{_{|I| \text{ odd}}} .
$$
This can be checked using \eqref{eq:H_I/H_J}.
\end{Remark}

\begin{Remark}
The image of $\Nef(X)$ under $\varphi$ is given by 
$$
		\Sigma \ = \ \varphi\big(\Nef(X)\big) \ = \  \left\{ 
		\begin{aligned}
		&H_{\{i\}}\geq 2, \ & i\in \{1, \dots, n+3\} \\
		&H_{\{i,j\}}\leq 3 , \ &  i, j\in \{1, \dots, n+3\}, \ i\neq j.
		\end{aligned}
		\right.
$$
\end{Remark}

\begin{Remark}
Formula \eqref{eq:phi}, together with the equations for the walls in the 
chamber decomposition of $\Delta$ defined by the hyperplane arrangement \eqref{eq:MCD_Delta},
allow us to find explicit inequalities defining the cones $\Eff(X)$, $\Mov(X)$, and $\Nef(\widetilde{X})$,
for any small $\Q$-factorial modification $\widetilde{X}$ of $X$.
\end{Remark}

\subsection{{Proof of Theorem~\ref{Thm:logFanos}}}
Let  $X^n_k$ be a blow-up of $\P^n$ at $k$ points in general position.
By \cite{Muk01} and \cite{CT},  $X^n_k$
is a Mori dream space if and only if one of the following holds:
\begin{itemize}
	\item[-] $n=2$ and $k\leq 8$.
	\item[-] $n=3$ and $k\leq 7$.
	\item[-] $n=4$ and $k\leq 8$.
	\item[-] $n>4$ and $k\leq n+3$.
\end{itemize}
We will show that in each of these cases $X^n_k$ is log Fano. 
In view of the  classification of del Pezzo surfaces and Example \ref{p3},
we may assume that $n\geq 4$.

Suppose that $k= n+3$, set $X:=X^n_{n+3}$, and follow the notation of the previous subsection. 
The center of the polytopes $\Pi$ and $\Delta$ is the point
$\left(\frac{1}{2}, \dots , \frac{1}{2}\right)=\varphi(-K_X)$.

When $n$ is even, this point is in the interior of a chamber of $\Pi$, namely
the chamber defined by: 
$$
\Sigma' \ = \ \left( \ H_I\ \geq \  \frac{n+2}{2} \ \right)_{|I|\not\equiv \frac{n+2}{2} \mod 2}.
$$
Let $X'$ be the small $\Q$-factorial modification of $X$ whose nef cone is the inverse image
of the chamber $\Sigma'$. 
Then $X'$ is a smooth Fano manifold with very interesting  geometry and symmetries. 
See \cite{Cinzia_2014} and references therein for several descriptions of $X'$.
By  Proposition~\ref{lfsmall}, $X$ is  log Fano.

When $n$ is odd, the point $\left(\frac{1}{2}, \dots , \frac{1}{2}\right)=\varphi(-K_X)$ lies in the 
intersection of the hyperplanes: 
$$
 \left( \ H_I\ = \  \frac{n+3}{2} \ \right)_{|I|\not\equiv \frac{n+3}{2} \mod 2}.
$$
Let $X'$ be the small $\Q$-factorial modification of $X$ whose nef cone is the inverse image
of some chamber $\Sigma'$ containing $\varphi(-K_X)$ in its boundary.
Then  $X'$ is a smooth projective variety with $-K_{X'}$ nef and big. 
By Lemma~\ref{lemma:wfano=>logfano}  and Proposition~\ref{lfsmall}, $X$ is  log Fano.

It follows from  \cite[Corollary 1.3]{GOST} that $X^n_k$ is log Fano for any  $k\leq n+3$.
 
The case $n=4$ and $k= 8$ can be treated in a similar way.
In \cite[Section 2]{MukaiADE}, Mukai describes the Mori chamber decomposition of $\Mov(X_8^4)$.
It follows from his description that $X_8^4$ admits a small $\Q$-factorial modification $X'$ which is a Fano manifold. 
Again we conclude that $X_8^4$ is log Fano by Proposition~\ref{lfsmall}. \qed


\section{Finding explicit divisors making $X^n_k$ log Fano for $k\leq n+2$} \label{section:k<n+3}

Throughout this section, we let $p_1,...,p_k\in\mathbb{P}^n$ be general points, $k\leq n+2$, and let $X^n_k$ be the blow-up of $\mathbb{P}^n$ at $p_1,...,p_k$.
We shall exhibit an integral divisor $D\subset X_k^n$ and rational numbers $\epsilon>0$ such that 
$\Delta = \epsilon D$ makes  $X^n_k$ log Fano.
In order to show that  $-(K_{X^n_k}+\Delta)$ is ample, we will use Proposition~\ref{mc2}.
To show that  $(X^n_k,\Delta)$ is klt, we will need an explicit log resolutions for this pair,
which we introduce next. 

\begin{Notation}\label{notation:resolution} 
For each $0\leq h\leq k-1$ and each subset $I=\{i_1<\dots<i_{h+1} \}\subset \{1,\dots, k\}$, 
consider the $h$-dimensional linear subspace $H_I^h = \left\langle p_{i_1},...,p_{i_{h+1}}\right\rangle\subset\mathbb{P}^n$.
Denote by $\mathcal{H}^h$ the collection of all such $h$-dimensional linear subspaces, and by 
$\rho_{h} = \binom{k}{h+1}$ its cardinality. 

Let $\pi:Y\to X^n_k$ be the blow-up of  the strict transforms of the lines in $\mathcal{H}^1$, 
followed by the blow-up of  the strict transforms of  the planes in $\mathcal{H}^2$, and so on,  in order of increasing dimension,
up to the  blow-up of  the strict transforms of the $(n-2)$-planes in $\mathcal{H}^{n-2}$.
For each $1\leq h\leq n-2$, denote by  $E_1^h,...,E^h_{\rho_h}\subset Y$ the exceptional divisors over the $\rho_h$
$h$-planes in $\mathcal{H}^h$.
We have 
\begin{equation}\label{eq:K_resolution}
K_Y \ = \ \pi^{*}K_{X^n_{k}}+\sum_{h=1}^{n-2}(n-h-1)(E^{h}_{1}+...+E^{h}_{\rho_h}).
\end{equation}
\end{Notation}

\begin{Remark}
For $k=n+1$, the variety $Y$ constructed above is the Losev-Manin moduli space, 
introduced in \cite{LM} as a toric compactification of $M_{0,n+3}$.
For $k=n+2$,  the construction above gives Kapranov's description of $Y=\overline M_{0,n+3}$ as an iterated blow-up of $\P^n$ (\cite{Ka}).
\end{Remark}


\subsection{Blow-ups of $\mathbb{P}^{n}$ in at most $n+1$ points}
For $k\leq n+1$, the variety $X_{k}^n$ is toric, 
and one can take $\Delta$ making $X^n_k$ log Fano to be a suitable combination of toric invariant divisors.
Alternatively, we  show that the divisor $\Delta$ can be chosen irreducible.
We work out the case $k=n+1$. When $k<n+1$, the boundary divisor can be taken to be the image 
of $\Delta$ under the natural morphism $X_{n+1}^n\to X_{k}^n$. 

\begin{Theorem} \label{lfn+1}
Let $D\subset X_{n+1}^n$ be the strict transform of a general member of
the linear system $\Gamma\subset |\mathcal{O}_{\mathbb{P}^n}(n)|$ of  the standard Cremona transformation of 
$\mathbb{P}^{n}$, centered at $p_1,...,p_{n+1}$.
For any $\frac{n-3}{n-2}<\epsilon< 1$ the divisor $-(K_{X_{n+1}^n}+\epsilon D)$ is ample, and the pair $(X_{n+1}^n,\epsilon D)$ is klt.
\end{Theorem} 

For the proof of Theorem~\ref{lfn+1}, we will need the following.

\begin{Lemma}\label{lemmacrem}
Let $\Gamma\subset |\mathcal{O}_{\mathbb{P}^n}(n)|$ be the linear system of the standard Cremona transformation of 
$\mathbb{P}^{n}$, centered at $p_1,...,p_{n+1}$, and $D\in \Gamma$ be a general member.
Let $H_I^h$ and $\pi:Y\rightarrow\mathbb{P}^n$ be as in Notation~\ref{notation:resolution}. 
Then the strict transform $\widetilde{D}$ of $D$ in $Y$ is smooth and transversal to all the exceptional divisors of $\pi$. 
Furthermore
$$
\mult_{H_I^h}D = n-h+1
$$
for any $h = 0,...,n-2$.
\end{Lemma}
\begin{proof}
By \cite[Theorem 1]{MM} the Cremona transformation induced by $\Gamma$  lifts to an automorphism of $Y$. 
This implies that $\widetilde{D}$  is smooth and transversal to all the exceptional divisors of $\pi$. 
In particular, $D$ is smooth away from  the union of the codimension two linear subspaces $H_I^{n-2}$'s.

We may assume that the $p_i$'s are the fundamental points of $\mathbb{P}^n$, and consider 
the element of the linear system $\Gamma$ given by:
$$
D_0 := \{x_0x_1...x_{n-1}+x_0x_1...x_{n-2}x_n+...+x_1x_2...x_n=0\}.
$$
Let $p\in H_I^h$ be a general point. Then one checks easily that 
$\mult_pD_0 =\mult_{H_I^h}D_0 = n-h+1$. 
To conclude, note that $\mult_{H_I^h}D \geq n-h+1$ for any $D\in \Gamma$.
\end{proof}

\begin{proof}[Proof of Theorem~\ref{lfn+1}]
With Notation~\ref{notation:generators}, we have 
$$
D \sim nH-(n-1)(E_1+...+E_{n-1}).
$$ 
Recall from Proposition~\ref{mc2} that the Mori cone of $X_{n+1}^n$ is generated by the classes $R_i$'s and $L_{i,j}$'s.
One computes 
$$
-(K_{X_{n+1}^n}+\epsilon D)\cdot R_i = n-1-\epsilon(n-1) \ \text{ and } -(K_{X_{n+1}^n}+\epsilon D)\cdot L_{i,j} = (n-1)\epsilon -n+3.
$$
Therefore $-K_{X_{n+1}^n}-\epsilon D$ is ample provided that $\frac{n-3}{n-2}<\epsilon< 1$.

Next we check when the pair $(X_{n+1}^n,\epsilon D)$ is klt.
Let  $\pi:Y\to X^n_{n+1}$ be the morphism introduced in Notation~\ref{notation:resolution}. 
By Lemma~\ref{lemmacrem} $\pi:Y\to X^n_{n+1}$ is a log resolution of $(X_{n+1}^n,\epsilon D)$, and 
$$
\pi^{*}(D)=\widetilde{D} + \sum_{h=1}^{n-2}(n-h-1)(E_1^{h}+...+E_{\rho_h}^{h}).
$$
Together with \eqref{eq:K_resolution}, this gives
$$
K_Y + \epsilon\widetilde{D}= \pi^*(K_{n+1}^n+\epsilon D)+\sum_{h=1}^{n-2}(n-h-1)(1-\epsilon)(E_1^{h}+...+E_{\rho_h}^{h}).
$$ 
Therefore the pair $(X_{n+1}^n,\epsilon D)$ is klt  for any $0\leq \epsilon< 1$.
\end{proof}


\subsection{Blow-ups of $\mathbb{P}^{n}$ in $n+2$ points}

In this subsection  we construct  divisors $\Delta$ making $X^n_{n+2}$ log Fano.
Note that $X^n_{n+2}$  is not toric.
We follow Notation~\ref{notation:resolution}, and denote by $H_1, \dots, H_{\rho_{n-1}}\subset \P^n$
the $\rho_{n-1}$ hyperplanes through $n$ of the $p_i$'s.

\begin{Theorem} \label{n+2}
Let $D\subset X_{n+2}^n$ be the strict transform of the divisor $H_1 + \dots + H_{\rho_{n-1}}$.
For any $\frac{2(n-3)}{(n+1)(n-2)}<\epsilon < \frac{2(n-1)}{n(n+1)}$ the divisor $-(K_{X_{n+2}^n}+\epsilon D)$ is ample, 
and the pair $(X_{n+2}^n,\epsilon D)$ is klt.
\end{Theorem} 

For the proof of Theorem~\ref{n+2}, we will need the following.

\begin{Lemma}\label{logres}
Let $D\subset X_{n+2}^n$ be the strict transform of the divisor $H_1 + \dots + H_{\rho_{n-1}}$.
Let  $\pi:Y\to X^n_{n+2}$ be the morphism introduced in Notation~\ref{notation:resolution}. 
Then $\pi:Y\to X^n_{n+2}$ is a log resolution of $(X_{n+2}^n, D)$, and 
$$
\pi^{*}( D)=\widetilde{D} + \sum_{h=1}^{n-2}\binom{n-h+1}{n-h-1}(E_1^{h}+...+E_{\rho_h}^{h}).
$$
\end{Lemma}
\begin{proof}
Note that at each step in the description of $Y$ as an iterated blow-up, the center of the blow-up is a disjoint union of 
smooth subvarieties. Moreover, the divisor $\Exc(\pi)\cup \widetilde{D} $ is simple normal crossing, and so $\pi:Y\to X^n_{n+2}$ is a log resolution of $(X_{n+2}^n, D)$.

Any  $H_I^h\in \mathcal{H}^h$ is contained in exactly $\binom{n-h+1}{n-h-1}$ of the $\rho_{n-1}$ hyperplanes
$H_i$'s. Thus $\mult_{H_I^h}D=\binom{n-h+1}{n-h-1}$, and the formula for $\pi^{*}( D)$ follows. 
\end{proof}

\begin{proof}[Proof of Theorem~\ref{n+2}]
Each point $p_i$ lies in exactly $\binom{n+1}{n-1} = \frac{1}{2}(n+1)n$ hyperplanes among the $H_i$'s, 
$1\leq i\leq \rho_{n-1}=\frac{1}{2}(n+2)(n+1)$.
So  we have, with  Notation~\ref{notation:generators},
$$
D \sim \frac{1}{2}(n+2)(n+1)H-\frac{1}{2}(n+1)nE_1-...-\frac{1}{2}(n+1)nE_{n+2}.
$$ 
Recall from Proposition~\ref{mc2} that the Mori cone of $X_{n+2}^n$ is generated by the classes $R_i$'s and $L_{i,j}$'s.
One computes 
$$
-(K_{X_{n+2}^n}+\epsilon D)\cdot R_i = (n-1-\frac{\epsilon}{2}(n+1)n) \ \text{ and } 
-(K_{X_{n+2}^n}+\epsilon D)\cdot L_{i,j} =  -n+3 +\frac{\epsilon}{2}(n+1)(n-2).
$$
Therefore $-K_{X_{n+2}^n}-\epsilon D$ is ample provided that $\frac{2(n-3)}{(n+1)(n-2)}<\epsilon < \frac{2(n-1)}{n(n+1)}$.

Next we check when the pair $(X_{n+2}^n,\epsilon D)$ is klt.
Let  $\pi:Y\to X^n_{n+2}$ be the morphism introduced in Notation~\ref{notation:resolution}. 
By Lemma~\ref{logres} $\pi:Y\to X^n_{n+2}$ is a log resolution of $(X_{n+2}^n,\epsilon D)$, and 
$$
\pi^{*}( D)=\widetilde{D} + \sum_{h=1}^{n-2}\binom{n-h+1}{n-h-1}(E_1^{h}+...+E_{\rho_h}^{h}).
$$
Together with \eqref{eq:K_resolution}, this gives
$$
K_Y + \epsilon\widetilde{D}= \pi^*(K_{n+2}^n+\epsilon D)+\sum_{h=1}^{n-2}\left((n-h-1)-\epsilon\binom{n-h+1}{n-h-1}\right)(E_1^{h}+...+E_{\rho_h}^{h}).
$$ 
Therefore the pair $(X_{n+2}^n,\epsilon D)$ is klt for any $0\leq \epsilon< \frac{2}{n}$.
\end{proof}


\section{Finding explicit divisors making $X^n_{n+3}$ log Fano} \label{section:k=n+3}

Throughout this section, let $p_1,...,p_{n+3}\in\mathbb{P}^n$ be general points, and let $X^n_{n+3}$ be the blow-up of $\mathbb{P}^n$ at $p_1,...,p_{n+3}$.
We shall exhibit integral divisors $D\subset X_{n+3}^n$ and rational numbers $\epsilon>0$ such that 
$\Delta = \epsilon D$ makes  $X^n_{n+3}$ log Fano.
In the previous cases, $D$ was taken as sum of strict transforms of hyperplanes through 
$n$ of the $n+3$ points. For $X^n_{n+3}$, we will also need  other extremal divisors $E_I\subset X^n_{n+3}$ introduced in Paragraph~\ref{Eff(X)}.
This will make the log resolution of $(X,\Delta)$ more complicated, and we will need to understand well how the divisors $E_I$'s intersect.
For this purpose, we start this section with some preliminaries on secant varieties of rational normal curves. 
Then we will  consider separately the cases $n = 2h+1$ odd, and $n = 2h$ even. 

\subsection{Preliminaries on secant varieties of rational normal curves}\label{subsec:sec}

Given an irreducible and reduced non-degenerate variety $X\subset\P^n$, and a positive integer $k\leq n$ we denote by $\sec_k(X)$ 
the \emph{$k$-secant variety} of $X$. This is the subvariety of $\P^n$ obtained as the closure of the union of all $(k-1)$-planes 
$\langle x_1,...,x_{k}\rangle$ spanned by $k$ general points of $X$. 
We will be concerned with the case when $X=C$ is a rational normal curve of degree $n$ in $\P^n$. 
The following proposition gathers some of the basic properties of the secant varieties $\sec_k(C)$ in this case.

\begin{Proposition}\label{sec1}
Let $C\subset\mathbb{P}^{n}$ be a rational normal curve of degree $n$, and let $k$ be an integer such that $1\leq k \leq\frac{n}{2}$.
Then the following statements hold.
\begin{enumerate}
	\item $\dim(\sec_{k}(C)) = 2k-1$ (see for instance  \cite[Proposition 11.32]{Har}).
	\item $\deg(\sec_{k}(C)) = \binom{n-k+1}{k}$ (see for instance  \cite[Theorem 12.16]{EH}).
	\item $\sec_k(C)$ is normal and $\Sing(\sec_{k}(C)) =\sec_{k-1}(C)$ (see for instance  \cite[Theorem 1.1]{Ve1}).
	\item  If $n = 2h$ is even, then for any $1\leq t< h$ we have 
		$$\mult_{\sec_{h-t}(C)}\sec_{h}(C) = t+1.$$ 
\end{enumerate} 
\end{Proposition}

\begin{proof}[Proof of (4)]
Suppose that $n = 2h$ is even, and consider the $(h+1)\times (h+1)$ matrix 
		\begin{equation}\label{eq:M_alpha}
		M_{h} = \left(\begin{matrix}
		x_0 & x_1  & \hdots & x_{h}\\ 
		x_1 & x_2  & \hdots & x_{h+1}\\ 
		\vdots & \vdots & \ddots & \vdots \\ 
		x_{h} & x_{h+1} & \hdots & x_{2h}
		\end{matrix}\right).
		\end{equation}
For any  $1\leq k\leq h$, the secant variety $\sec_k(C)$ can be described as the determinantal variety:
		$$
		\sec_k(C) \ = \ \big\{ \rank(M_h) \leq k  \big\}.
		$$ 
(See for instance \cite[Proposition 9.7]{Har}).
In particular, $\sec_h(C)\subset\mathbb{P}^{2h}$ is the degree $h+1$ hypersurface defined by the polynomial 
$F := \det(M_h)$. 
For each $j \in \{0,...,2h\}$,
let $\{M_i^j\}$ be the set of  $h\times h$ minors of $M_h$ produced by erasing in $M_h$ a row and a column meeting in an entry of type $x_j$
Denote by  $\rho_j$ be the number of such minors. Then 
$$
\frac{\partial F}{\partial x_j} = \sum_{i=1}^{\rho_j}\alpha_i^j\det(M_i^j),
$$ 
for suitable $\alpha_i^j\neq 0$.
Inductively, we see that for any $1\leq t< h$ the partial derivatives of order $t$ of $F$ are linear combinations of determinants of 
$(h+1-t)\times (h+1-t)$ minors of $M_h$.
The vanishing of such determinants defines $\sec_{h-t}(C)$, while the vanishing of the of determinants of the
$(h-t)\times (h-t)$ minors of $M_h$ defines $\sec_{h-t-1}(C)\subsetneq \sec_{h-t}(C)$.
Therefore, there is at least one partial derivative of order $t+1$ of $F$ not vanishing on $\sec_{h-t}(C)$. 
This means that $\mult_{\sec_{h-t}(C)}\sec_{h}(C) = t+1$ for any $1\leq t< h$. 
\end{proof}

The following proposition is just a particular instance of \cite[Theorem 1]{Be}. The general statement for smooth curves embedded via a $2h$-very ample line bundle can be found in \cite[Theorem 3.1]{Ve} as well.

\begin{Proposition}\label{sec2}
Let $C\subset\mathbb{P}^{n}$ be a rational normal curve of  degree $n$, and set $h:=\left\lfloor \frac{n}{2}\right\rfloor$. 
Consider the following sequence of blow-ups:
\begin{itemize}
\item[-] $\pi_1:X_1\rightarrow\mathbb{P}^{n}$ the blow-up of $C$,
\item[-] $\pi_2:X_2\rightarrow X_{1}$ the blow-up of the strict transform of $\sec_2(C)$,\\
\vdots
\item[-] $\pi_{h}:X_{h}\rightarrow X_{h-1}$ the blow-up of the strict transform of $\sec_{h}(C)$.
\end{itemize}
Let $\pi:X\rightarrow\mathbb{P}^{n}$ be the composition of these blow-ups. Then, for any $k\leq h$ the strict transform of $\sec_{k}(C)$ in $X_{k-1}$ is smooth and transverse to all exceptional divisors. In particular $X$ is smooth and the exceptional locus of $\pi$  is a simple normal crossing divisor.
\end{Proposition}

\begin{Notation}\label{notation:join&secs}
Let $p_1,...,p_{n+3}\in\mathbb{P}^n$ be general points, and let $C\subset\mathbb{P}^{n}$ be the unique rational normal curve of  
degree $n$ through these points. 
Given $1\leq m\leq n$, $I=\{i_1<\dots<i_m \}\subset \{1,\dots, n+3\}$, and a positive integer $k$  such that $0\leq k\leq \frac{n-m}{2}$,
we consider the following variety of dimension $d=2k-1+m$:
$$
Y^d_I \ := \  \Jo\big( \ \langle p_{i_1},\dots , p_{i_m}\rangle \ , \ \sec_{k}(C) \ \big).
$$
Alternatively, $Y^d_I$ can be defined as follows. 
Let $\pi_I:\P^n\map \P^{n-m}$ be the projection from the linear space $\langle p_{i_1},\dots , p_{i_m}\rangle$. 
Let $C_I\subset \P^{n-m}$ be the image of $C$ under $\pi_I$. It is the  the unique rational normal curve of  
degree $n-m$ through the points $\pi(p_j)$, $j\not\in I$.
Then $Y^d_I $ is the cone with vertex $\langle p_{i_1},\dots , p_{i_m}\rangle$ over $\sec_{k}(C_I)$. 

By convention, when $k=0$, we set $Y^{m-1}_I := \langle p_{i_1},\dots , p_{i_m}\rangle$.
\end{Notation}

Fix $I=\{i_1<\dots<i_m \}\subset \{1,\dots, n+3\}$, with $m\leq n$.
Given $k$  such that $0\leq k\leq \frac{n-m}{2}$, set $d:=2k-1+m$. 
By Proposition~\ref{sec1}, we have 
\begin{equation} \label{eq:degY_I}
\deg(Y^d_I ) = \binom{n-m-k+1}{k} \ \text{ and } \ \Sing(Y^d_I ) =Y^{d-2}_I 
\end{equation} 
Moreover, if $n-m$ is even and $d_1=2k_1-1+m>2k_2-1+m=d_2$, then $Y^{d_{_2}}_{I}\subset Y^{d_{_1}}_{I}$ and
\begin{equation} \label{eq:multY_I}
\mult_{Y^{d_{_2}}_{I}}Y^{d_{_1}}_{I} = \frac{d_1-d_2}{2}+1.
\end{equation}

We also have analogs of Proposition~\ref{sec2} for sequences of blow-ups of $Y_I^{d}$, for $ |I|-1\leq d\leq n-1$. 
More precisely:

\begin{Proposition}\label{lcones}
Let $C\subset\mathbb{P}^{n}$ be a rational normal curve of  degree $n$, $p_{1},\dots , p_{m}\in C$ distinct points, with $1\leq m\leq n$,
and set $h:=\left\lfloor \frac{n-m}{2}\right\rfloor$.
Consider the following sequence of blow-ups:
\begin{itemize}
\item[-] $\pi_1:X_1\rightarrow\mathbb{P}^{n}$ the blow-up of $Y^{m-1}_I := \langle p_{1},\dots , p_{m}\rangle$,
\item[-] $\pi_2:X_2\rightarrow X_{1}$ the blow-up of the strict transform of $Y^{m+1}_I$,\\
\vdots
\item[-] $\pi_{h}:X_{h}\rightarrow X_{h-1}$ the blow-up of the strict transform of $Y^{m+2h-1}_I$.
\end{itemize}
Let $\pi:X\rightarrow\mathbb{P}^{n}$ be the composition of these blow-ups. Then, for any $k\leq h$ the strict transform of 
$Y^{m+2k-1}_I$ in $X_{k-1}$ is smooth and transverse to all exceptional divisors. 
\end{Proposition}

Proposition~\ref{lcones} follows easily from Proposition~\ref{sec2}. 
In the next sections, we will blow-up varieties of type $Y^d_I$ for several subsets $I\subset \{1,\dots, n+3\}$, in a suitable order. 
In order to show the smoothness and transversality of the strict transforms of the $Y^d_I$ 's in the intermediate blow-ups, we 
will need the following result.

\begin{Proposition}\label{ordsing}
Let $W\subsetneq Z\subsetneq X$ be smooth projective varieties, and let $Y\subset X$ be a projective variety such that $\Sing(Y) = Z$ 
and $Y$ has ordinary singularities along $Z$. Let $\pi_W:X_W\rightarrow X$ be the blow-up of $W$, and denote by $Z_W$ and $Y_W$ 
the strict transforms of $Z$ and $Y$, respectively. Then $\Sing(Y_W) = Z_W$ and $Y_W$ has ordinary singularities along $Z_W$.
\end{Proposition}

\begin{proof}
Denote by $E_W$ the exceptional divisor of $\pi_W$.
Then $\pi_W^{-1}(Z) = Z_W\cup E_W$.  
Let $\pi_{Z_W}:X_{Z_W}\rightarrow X_W$ be  the blow-up  of  $X_W$ along  $Z_W$, with exceptional 
  divisor $E_{Z_W}$.

We claim that the composite morphism $\pi_W\circ \pi_{Z_W}: X_{Z_W}\rightarrow X$ is isomorphic to the blow-up
$\pi_Z:X_Z\rightarrow X$ of $X$ along $Z$, followed by the blow-up of $X_Z$ along $\pi_{Z}^{-1}(W)$.
Indeed, by the universal property of the blow-up (\cite[Proposition 7.14]{Hart}), there exits a unique morphism
$f:X_{Z_W}\rightarrow X_Z$  making the following diagram commute. 
  \[
  \begin{tikzpicture}[xscale=2.1,yscale=-1.2]
    \node (A0_0) at (0, 0) {$X_{Z_W}$};
    \node (A0_1) at (1, 0) {$X_Z$};
    \node (A1_0) at (0, 1) {$X_W$};
    \node (A1_1) at (1, 1) {$X$};
    \path (A0_0) edge [->]node [auto] {$\scriptstyle{f}$} (A0_1);
    \path (A1_0) edge [->]node [auto] {$\scriptstyle{\pi_W}$} (A1_1);
    \path (A0_1) edge [->]node [auto] {$\scriptstyle{\pi_Z}$} (A1_1);
    \path (A0_0) edge [->,swap]node [auto] {$\scriptstyle{\pi_{Z_W}}$} (A1_0);
  \end{tikzpicture}
  \]
  Note that all varieties in this diagram are smooth. 
  Since $Z$ and $W$ are smooth, the intersection $Z_W \cap {E}_W\subset X_W$ is smooth. Thus, any normal direction of $Z_W$ 
  in $X_W$ at a point $p\in Z_W \cap {E}_W$ is the image of a normal direction at $p$ of $Z_W \cap {E}_W$ in $E_W$.
  In other words, the inverse image of $W$ in $X_{Z_W}$ consists of the strict transform $\widetilde{E}_W$ of $E_W$ in $X_{Z_W}$.
  Therefore, the inverse image of the smooth variety $\pi_{Z}^{-1}(W)$ in $X_W$ is precisely $\widetilde{E}_W$. 
  Using the the universal property of the blow-up, and comparing the Picard number of these smooth varieties, 
  we conclude that $f:X_{Z_W}\rightarrow X_Z$ is the blow-up of $X_Z$ along $\pi_{Z}^{-1}(W)$, proving the claim. 
  
  Next we prove that $\Sing(Y_W) = Z_W$. Clearly $Z_W \subset \Sing(Y_W)$.
  Suppose that this inclusion is strict. 
  Then the  strict transform $Y_{Z_W}$ of $Y_W$ in $X_{Z_W}$ is singular.
  Since  $f:X_{Z_W}\rightarrow X_Z$ is a smooth blow-up, $f(Y_{Z_W})\subset X_Z$ is singular as well.
  But notice that $f(Y_{Z_W})\subset X_Z$ is  the strict transform of $Y\subset X$ via $\pi_Z$.
  Since $\Sing(Y) = Z$ and $Y$ has ordinary singularities along $Z$, the blow-up $\pi_Z$ resolves the singularities of $Y$. 
  This contradiction shows that $\Sing(Y_W) = Z_W$.
  Moreover, since $Y$ has ordinary singularities along $Z$, the intersection of its strict transform $Y_Z$ with the exceptional 
  divisor $E_Z$ of $\pi_Z$ is transverse. 
  This implies that the intersection $Y_{Z_W}\cap E_{Z_W}$ is also transverse, i.e., 
  $Y_W$ has ordinary singularities along $Z_W$.
 \end{proof}

We end this section by describing  the intersection of some of the $Y^d_I$'s. This can be computed using elementary projective geometry.
In what follows we adopt the following notation. 
Given two finite sets $I$ and $J$, we define their distance to be
$$
d(I,J) \ := \ \big| (I\cup J)\setminus (I\cap J) \big|. 
$$
We start by intersecting varieties $Y^d_I$'s of the same dimension.

\begin{Proposition}\label{intersections}
Let the assumptions and notation be as in Notation~\ref{notation:join&secs}.
Let  $I_1, I_2\subset \{1,\dots, n+3\}$ be subsets with cardinality $m_1$ and $m_2$, respectively,
and suppose that $I_1\cap I_2 = \emptyset$.
Let $k_1$ and $k_2$ be integers such that 
$0\leq k_i\leq \frac{n-m_i}{2}$, $i=1,2$, and 
$m_1+2k_1-1=m_2+2k_2-1=:d$.
Set $s=\frac{m_1+m_2}{2}$ and suppose that $d\leq n-s$.
Then 
	$$
	Y_{I_1}^d \ \cap \  Y_{I_2}^{d} \ = \ \bigcup_J Y_J^{d-s},
	$$ 
where the union is taken over all subsets $J \subset I_1\cup I_2$ satisfying $d(I_i,J)=s$ for $i=1,2$.

Moreover, for a general point in any irreducible component of the above intersections, the intersection is transverse.
\end{Proposition}

\begin{proof}
We note that the assumptions of the theorem imply that $d=k_1+k_2+s-1$ and $m_1-m_2=2(k_2-k_1)$.

\medskip 

Let $J \subset I_1\cup I_2$ be such that  $d(I_i,J)=s$ for $i=1,2$.
We shall prove that $Y_J^{d-s}\subset Y_{I_1}^d  \cap  Y_{I_2}^{d}$.
Write $J=J_1\cup J_2$, where $J_i\subset I_i$,  $i=1,2$, set
$\ell_i:=|J_i|$ ,  $i=1,2$,  and $\ell=|J|=\ell_1+\ell_2$.
The assumption that $d(I_i,J)=s$ for $i=1,2$ implies that $k_2-k_1=\ell_1-\ell_2$.
We set $k:= k_2-\ell_1=k_1-\ell_2$, and note that $d-s=\ell+2k-1$. 

Let $x\in Y_J^{d-s}$.
Then there exists a point $q\in \sec_{k}(C)$ such that $x\in \langle \ q, p_i \ | \ i\in J \ \rangle\cong \P^\ell$.
The following two linear subspaces of this $\P^\ell$
$$
\langle \ x, p_i \ | \ i\in I_1 \ \rangle\cong \P^{\ell_1} \ \text{ and } \ \langle \ q, p_i \ | \ i\in I_2 \ \rangle\cong \P^{\ell_2} 
$$ 
have complementary dimensions. Hence there exists a point 
$$
z \ \in \ \langle \ x, p_i \ | \ i\in J_1 \ \rangle \ \cap \ \langle \ q, p_i \ | \ i\in J_2 \ \rangle.
$$
In particular, $z\in \sec_{k+\ell_2}(C)$. Since $k+\ell_2=k_1$, we conclude that 
$x\in Y_{I_1}^d$.
Similarly we show that $x\in Y_{I_2}^d$.

Now assume that $x$ is a general point of $Y_J^{d-s}$.
Keeping the same notation as above, we will prove now that $Y_{I_1}^d$ and $Y_{I_2}^{d}$
intersect transversely at $x$.
This amounts to proving that $T_x\big(Y_{I_1}^d\big)\cap T_x\big(Y_{I_2}^{d}\big)=T_x\big(Y_J^{d-s}\big)$.
By Terracini's Lemma \cite{Te}, we have 
\begin{align*}
T_x\big(Y_{I_1}^d\big) \ &= \ \langle \  \langle p_i \ | \ i\in I_1 \   \rangle, \   \langle \ T_{q_i}C \ | \ 1\leq i\leq k \   \rangle, 
\ \langle \ T_{p_i}C \ | \ i\in J_2 \   \rangle    \  \rangle, \\
T_x\big(Y_{I_2}^{d}\big)  \ &= \ \langle \  \langle p_i \ | \ i\in I_2 \   \rangle, \   \langle \ T_{q_i}C \ | \ 1\leq i\leq k \   \rangle, 
\ \langle \ T_{p_i}C \ | \ i\in J_1 \   \rangle    \ \rangle, \\
T_x\big(Y_J^{d-s}\big)  \ &= \ \langle \  \langle p_i \ | \ i\in J \   \rangle, \   \langle \ T_{q_i}C \ | \ 1\leq i\leq k \   \rangle  \  \rangle,
\end{align*}
where $q_1, \dots , q_k\in C$ are such that $q\in  \langle \ q_i \ | \ 1\leq i\leq k \   \rangle$.

Consider the linear subspaces:
\begin{align*}
L_1 \ &:= \ \langle \  \langle p_i \ | \ i\in I_1 \   \rangle, \ \langle \ T_{p_i}C \ | \ i\in J_2 \   \rangle    \  \rangle, \\
L_2  \ &:= \ \langle \  \langle p_i \ | \ i\in I_2 \   \rangle, \ \langle \ T_{p_i}C \ | \ i\in J_1 \   \rangle    \ \rangle, \\
L \ &:= \ \langle \  \langle p_i \ | \ i\in J \   \rangle \  \rangle \ \subset \ L_1\cap L_2 .
\end{align*}
We have that $\dim(\langle L_1, L_2 \rangle)\leq m_1+m_2 +\ell -1$, and equality holds if and only if $L_1\cap L_2 = L$.
On the other hand, note that $L$ intersects $C$ in at least $m_1+m_2 +\ell$ points, counted with multiplicity. 
Therefore we must have $\dim(\langle L_1, L_2 \rangle)= m_1+m_2 +\ell -1$, and $L_1\cap L_2 = L$.
It follows from the description of the tangent spaces above that $T_x\big(Y_{I_1}^d\big)\cap T_x\big(Y_{I_2}^{d}\big)=T_x\big(Y_J^{d-s}\big)$.

\

It remains to prove that $Y_{I_1}^d \ \cap \  Y_{I_2}^{d} \subset \bigcup_J Y_J^{d-s}$.
Write $\{p_i \ | \ i\in I_1\} = \{x_1,\dots, x_{m_1}\}$ and $\{p_i \ | \ i\in I_2\} = \{y_1,\dots, y_{m_2}\}$. 
Suppose that $x\in Y_{I_1}^d  \cap  Y_{I_2}^{d}$. This means that 
there exist points $z_1, \dots, z_{k_1}, w_1, \dots, w_{k_2}\in C$ such that:
\begin{align*}
\langle x_{1},\dots , x_{m_1}\rangle \cap \langle z_{1},\dots , z_{k_1} \rangle \ = \ \emptyset \ = & \
\langle y_{1},\dots , y_{m_2}\rangle \cap \langle w_{1},\dots , w_{k_2}\rangle, \ \text{ and } \\
x\ \in \  \langle x_{1},\dots , x_{m_1}, z_{1},\dots , z_{k_1} \rangle &\cap \langle y_{1},\dots , y_{m_2}, w_{1},\dots , w_{k_2}\rangle. 
\end{align*}
The assumption that $d\leq n-s$ implies that $m_1+m_2+k_1+k_2\leq n+1$, and thus 
\begin{align*}
\langle x_{1},\dots , x_{m_1}, z_{1},\dots , z_{k_1} \rangle \ &\cap \ \langle y_{1},\dots , y_{m_2}, w_{1},\dots , w_{k_2}\rangle \ =\ \\
\langle \  \{x_{1},\dots , x_{m_1}, z_{1},\dots , z_{k_1}\}  \ & \cap \ \{ y_{1},\dots , y_{m_2}, w_{1},\dots , w_{k_2}\} \  \rangle.
\end{align*}
By relabeling the points if necessary, we may write, for suitable integers $s_1$, $s_2$ and $r$:
\begin{align*}
 \{x_{1},\dots , x_{s_1}\} \ &= \ \{x_{1},\dots , x_{m_1}\} \ \cap \{w_{1},\dots , w_{k_2}\} \\
 \{y_{1},\dots , y_{s_2}\} \ &= \ \{y_{1},\dots , y_{m_2}\} \ \cap \{z_{1},\dots , z_{k_1}\} \\
 \{z_{1}=w_1,\dots , z_{r}=w_r\} \ &= \ \{z_{1},\dots , z_{k_1}\} \ \cap \{w_{1},\dots , w_{k_2}\}.
 \end{align*}
Note that $s_i+r\leq k_j$, $\{i,j\}=\{1,2\}$, and we have
\begin{equation}\label{x,y,z,w}
x\ \in \  \langle \  x_{1},\dots , x_{s_1}, y_{1},\dots , y_{s_2}, z_{1},\dots , z_{r}  \ \rangle.
\end{equation}
Let $J_0\subset I_1\cup I_2$ be the subset of indices corresponding to the subset
$\{ x_{1},\dots , x_{s_1}, y_{1},\dots , y_{s_2}\}\subset \{p_1,...,p_{n+3}\}$.
Note that  $d(J_0,I_i)=m_i-s_i+s_j$, for $\{i,j\}=\{1,2\}$.
In particular we have  
$$
d(J_0,I_1)+d(J_0,I_2)=2s.
$$

Suppose first  that $d(J_0,I_1)=d(J_0,I_2)=s$. 
It follows from \eqref{x,y,z,w} that 
$$
x\in \Jo\big( \ \langle p_{i}  \ | \ i\in J_0 \ \rangle \ , \ \sec_{r}(C) \ \big).
$$
Since $s_i+r\leq k_j$, $\{i,j\}=\{1,2\}$, we get that
$$
|J_0| +2r-1 \ = \ s_1+s_2 + 2r-1 \ \leq \ k_1+k_2-1 \ = \ d-s.
$$
Hence $x\in Y_{J_0}^{d-s}$.

From now on we consider the case when $d(J_0,I_1)\neq d(J_0,I_2)$.
Without lost of generality, we assume that 
$$
d(J_0,I_1)-d(J_0,I_2)\ = \ m_1-m_2+2s_2 -2s_1\ > \ 0.
$$
We will modify the subset $J_0\subset I_1\cup I_2$ by adding points of $I_1\setminus J_0$ 
or removing points of $I_2\cap J_0$ to obtain another subset $J\subset I_1\cup I_2$ satisfying $d(I_i,J)=s$ for $i=1,2$.
Note that if $i\in I_1\setminus J_0$, then $d(J_0\cup\{i\},I_1)=d(J_0,I_1)-1$ and $d(J_0\cup\{i\},I_2)=d(J_0,I_2)+1$. 
Similarly, if $i\in I_2\cap J_0$, then $d(J_0\setminus\{i\},I_1)=d(J_0,I_1)-1$ and $d(J_0\setminus\{i\},I_2)=d(J_0,I_2)+1$. 
So we have to modify $J_0$ by adding or removing exactly $\frac{m_1-m_2}{2} +s_2-s_1$ points of the appropriate $I_i$.

Suppose first that $\big|I_1\setminus J_0\big|=m_1-s_1\geq \frac{m_1-m_2}{2} +s_2-s_1$. This is equivalent to the inequality 
$s\geq s_2$. We construct  $J_1\subset I_1\cup I_2$ by adding to $J_0$ exactly $\frac{m_1-m_2}{2} +s_2-s_1$ points of 
$I_1\setminus J_0$. Then $d(I_i,J_1)=s$ for $i=1,2$, and it follows from \eqref{x,y,z,w} that 
$$
x\in \Jo\big( \ \langle p_{i}  \ | \ i\in J_1 \ \rangle \ , \ \sec_{r}(C) \ \big).
$$
Since $s_2+r\leq k_1$,  we get that
$$
|J_1| +2r-1 \ = \ (k_2-k_1+2s_2) + 2r-1 \ \leq \ k_1+k_2-1 \ = \ d-s.
$$
Hence $x\in Y_{J_1}^{d-s}$.

Next we suppose that $s<s_2$. Let $I_2'\subset I_2$ be the  subset of indices corresponding to the subset $\{  y_{1},\dots , y_{s}\}$, and set $J_2:=I_1\cup I_2'$.
Then $d(I_i,J_2)=s$ for $i=1,2$, and it follows from \eqref{x,y,z,w} that 
$$
x\in \Jo\big( \ \langle p_{i}  \ | \ i\in J_2 \ \rangle \ , \ \sec_{r+s_2-s}(C) \ \big).
$$
Since $s_2+r\leq k_1$,  we get that
$$
|J_2| +2(r+s_2-s)-1 \ = \ m_1 + 2(r+s_2)-s-1 \ \leq \ m_1+2k_1-1-s \ = \ d-s.
$$
Hence $x\in Y_{J_2}^{d-s}$.

\end{proof}


\subsection{The odd case $n = 2h+1$} \label{subset:odd}

In this subsection  we construct  divisors $\Delta$ making $X^n_{n+3}$ log Fano
when $n = 2h+1$ is odd.
We follow Notation~\ref{notation:join&secs}.
For each $1\leq i \leq 3$, let $\Delta_i \subset X_{n+3}^n$ be the strict transform of the divisor $Y^{2h}_i\subset \P^n$,
and  denote by $H_{4,...,n+3} \subset X_{n+3}^n$ the strict transform of
the hyperplane $ \left\langle p_4,...,p_{n+3}\right\rangle\subset\mathbb{P}^{n+3}$.

\begin{Theorem}\label{mainodd}
Let $n = 2h+1\geq 5$ be  an odd integer. 
Set 
$$
D:=\Delta_1\cup\Delta_2\cup\Delta_3\cup  H_{4,...,n+3}\subset X_{n+3}^n.
$$
For any $\frac{2h-2}{3h-2}<\epsilon <\frac{2h}{3h+1}$ the divisor $-(K_{X_{n+3}^n}+\epsilon D)$ is ample, 
and the pair $(X_{n+3}^n,\epsilon D)$ is klt.
\end{Theorem}

For the proof of Theorem~\ref{mainodd}, we will need the following.

\begin{Proposition}\label{logodd}
Let the assumptions be as in Theorem~\ref{mainodd}, and follow Notation~\ref{notation:join&secs}.
For $0\leq m\leq n-3$, we define a modification $X_m$ of  $X_{n+3}^n$ recursively as follows:
\begin{itemize}
\item[-] $X_0=X_{n+3}^n$,
\item[-] $X_{2k+1}$ is the blow-up of $X_{2k}$ along the strict transforms of $\sec_{k+1}(C)$, and of the $Y_{i,j}^{2k+1}$'s ($0\leq k\leq h-2$),
\item[-] $X_{2k}$ is the blow-up of $X_{2k-1}$ along the strict transforms of the $Y_i^{2k}$'s, and of $Y_{1,2,3}^{2k}$ ($1\leq k\leq h-1$),
\item[-] $X_{n-2}$ is the blow-up  of $X_{n-3}$ along the strict transform of $\sec_h(C)$.
\end{itemize}

Then, for any $0\leq m\leq n-3$, the center of the blow-up $X_{m+1}\rightarrow X_m$ is a disjoint union of smooth subvarieties, 
all transverse to the exceptional divisors of $X_m \rightarrow X_0$. 
Moreover, the composition $\pi:X_{n-2}\rightarrow X_{n+3}^n$ of these blow-ups is a log resolution of the pair $(X_{n+3}^n,D)$.
\end{Proposition}

\begin{proof}
We will prove the result by induction on $m$. The statement is clearly try for $m=0$. 
For simplicity of notation we will denote by $\widetilde{Z}$ the strict transform of a subvariety $Z\subset X_{n+3}^n$ in any $X_m$.

Suppose that the statement is true $m=2k$. We will show that it holds for $X_{2k+1}$ and $X_{2k+2}$. 
We start with the  following observation. 
Let $I\subset \{1,\dots, n+3\}$ be such that either $|I|\in \{0,1,2\}$ or $I=\{1,2,3\}$, $0\leq k\leq \frac{n-|I|}{2}$,
and $d=2k-1+|I|$. 
Then, for any $m<d-1$, each component of the center of the blow-up $X_{m+1}\rightarrow X_m$ is either contained in  $\Sing(Y_I^d)$,
or is disjoint from it. 
This allows us to apply  Proposition~\ref{ordsing}, together with Propositions~\ref{sec2} and \ref{lcones}, and conclude by induction that 
the following holds.
\begin{itemize}
\item[-] The subvarieties  
$\widetilde{Y}_i^{2k+2}\subset X_{2k+1}$, $1\leq i\leq 3$,  are  smooth and transverse to the exceptional divisors over $X_0$.
\item[-]  The subvarieties  
$\widetilde{\sec_{k+2}(C)},\: \widetilde {Y}_{i,j}^{2k+3}\subset X_{2k+2}$, $1\leq i < j \leq 3$, are smooth and transverse to the exceptional divisors over $X_0$.
\end{itemize}

Next we show that  the $\widetilde {Y}_i^{2k+2}$'s and $\widetilde{Y}_{1,2,3}^{2k+2}$ are pairwise disjoint in $X_{2k+1}$, 
and similarly for $\widetilde{\sec_{k+2}(C)}$ and the $ \widetilde{Y}_{i,j}^{2k+3}$'s in  $X_{2k+2}$.

Consider the blow-up  $X_{2k+1} \rightarrow X_{2k}$.
By Proposition~\ref{intersections}, on $X_{2k}$ we have
 	$$
	\widetilde{Y}_i^{2k+2} \ \cap \  \widetilde{Y}_j^{2k+2} \ = \ \widetilde{\sec_{k+1}(C)} \ \cup \ \widetilde{Y}_{i,j}^{2k+1},  
	\quad \widetilde{Y}^{2k+2}_{i}\cap \widetilde{Y}^{2k+2}_{i,r,s} = \widetilde{Y}_{i,r}^{2k+1}\cup\widetilde{Y}_{i,s}^{2k+1}.
	$$ 
By the induction hypothesis, $\widetilde{\sec_{k+1}(C)}$ and $\widetilde{Y}_{i,j}^{2k+1}$ are smooth and  disjoint. 
So the intersections are everywhere transverse. 
We conclude that on $X_{2k+1}$, which is obtained from $X_{2k}$ by blowing-up $ \widetilde{\sec_{k+1}(C)}$ and $\widetilde{Y}_{i,j}^{2k+1}$,
the $\widetilde{Y}_i^{2k+2}$'s and $\widetilde{Y}_{1,2,3}^{2k+2}$ are pairwise disjoint. 

Now consider the blow-up $X_{2k+2}\rightarrow X_{2k+1}$. By Proposition~\ref{intersections}, on $X_{2k+1}$ we have
	$$
	\widetilde{\sec_{k+2}(C)} \cap  \widetilde{Y}_{i,j}^{2k+3}  \ = \ \widetilde{Y}_{i}^{2k+2}  \cup \widetilde{Y}_{j}^{2k+2}, \quad
	\widetilde{Y}_{i,j}^{2k+3} \cap  \widetilde{Y}_{i,r}^{2k+3} \  = \  \widetilde{Y}_i^{2k+2}   \cup \widetilde{Y}_{i,j,r}^{2k+2}.
	$$
By the induction hypothesis, the $ \widetilde{Y}_{i}^{2k+2}$'s  and $\widetilde{Y}_{1,2,3}^{2k+2}$ are smooth and  pairwise disjoint. 
So the intersections are everywhere transverse. 
We conclude that on $X_{2k+2}$, which is obtained from $X_{2k+1}$ by blowing-up  the $ \widetilde{Y}_{i}^{2k+2}$'s and $\widetilde{Y}_{1,2,3}^{2k+2}$,
the varieties $\widetilde{\sec_{k+2}(C)}$ and the $\widetilde{Y}_{i,j}^{2k+3}$'s are pairwise disjoint.

As before, we have that the divisors  $\widetilde{H}_{4,...,n+3}$, $\widetilde{\Delta}_1$, $\widetilde{\Delta}_2$ and $\widetilde{\Delta}_3$ 
on $X_{n-3}$ are smooth and transverse to the exceptional divisors over $X_0$, and their intersection are pairwise smooth and everywhere transverse. 
 By Proposition \ref{intersections} we have  
	$$
	\widetilde{\Delta}_1 \ \cap \ \widetilde{\Delta}_2 \ \cap \ \widetilde{\Delta}_3 \ = \ \widetilde{\sec_{h}(C)}.
	$$
So, after the blow-up $X_{n-2}\rightarrow X_{n-3}$ of $\widetilde{\sec_{h}(C)}$, we get a log resolution of $(X_{n+3}^n,D)$. 
\end{proof}

\begin{proof}[Proof of Theorem~\ref{mainodd}]
With  Notation~\ref{notation:generators}, we have 
$$
D = \Delta_1 + \Delta_2 + \Delta_3 +  H_{4,...,n+3} \sim (3h+4)H-(3h+1)(E_1+...+E_{n+3}).
$$
Recall from Proposition~\ref{mc2} that the Mori cone of $X_{n+3}^n$ is generated by the classes $R_i$'s and $L_{i,j}$'s.
One computes 
$$
-(K_{X_{n+3}^n}+\epsilon D)\cdot R_i = 2h-\epsilon (3h+1) \ \text{ and } 
-(K_{X_{n+3}^n}+\epsilon D)\cdot L_{i,j} =  \epsilon (3h-2)-2h+2.
$$
Therefore $-K_{X_{n+3}^n}-\epsilon D$ is ample provided that $\frac{2h-2}{3h-2}<\epsilon <\frac{2h}{3h+1}$.

Next we check when the pair $(X_{n+3}^n,\epsilon D)$ is klt.
Let  $\pi:\widetilde{X} :=X_{n-2} \to X^n_{n+3}$ be the log resolution of $(X_{n+3}^n,\epsilon D)$ introduced in Proposition~\ref{logodd} above. 
We have
$$
K_{\widetilde{X}} = \pi^{*}K_{X^n_{n+3}}+\sum_{k=1}^h(n-2k)E_{\sec_k(C)}+\sum_{k=1}^{h-1}(n-2k)\sum_{i,j}E_{Y_{i,j}^{2k-1}}+
\sum_{k=1}^{h-1}(n-2k-1)(\sum_{i}E_{Y_i^{2k}}+E_{Y_{1,2,3}^{2k}}).
$$
Here we denote by $E_Y$ the exceptional divisor with center $Y\subset \P^n$. In order to compute discrepancies, we will compute the the multiplicities of the $Y^{2h}_i$'s along the images in $\P^n$ of the subvarieties blown-up by $\pi$.
By Proposition \ref{sec1} we have $\mult_{\sec_k(C)}\sec_{h}(C) = h-k+1$. Moreover, $\mult_{\sec_k(C)}Y^{2h}_r = h-k+1$,
$$
\mult_{Y_{i,j}^{2k-1}}Y^{2h}_r = 
\left\lbrace\begin{array}{ll}
\mult_{\sec_{k}(C)}\sec_h(C) = h-k+1 & \rm{if} \: \textit{r}\in\{\textit{i,j}\}, \\ 
\mult_{\sec_{k+1}(C)}\sec_h(C) = h-k & \rm{if} \: \textit{r}\notin\{\textit{i,j}\},
\end{array}\right.
$$
$$
\mult_{Y_{i}^{2k}}Y^{2h}_r = 
\left\lbrace\begin{array}{ll}
\mult_{\sec_{k}(C)}\sec_h(C) = h-k+1 & \rm{if} \: \textit{r}=\textit{i}, \\ 
\mult_{\sec_{k+1}(C)}\sec_h(C) = h-k & \rm{if} \: \textit{r}\neq \textit{i},
\end{array}\right.
$$
and $\mult_{Y_{1,2,3}^{2k}}Y^{2h}_r =\mult_{\sec_{k+1}(C)}\sec_{h}(C) = h-k$ for for $k=1,...,h-1$. Let $\Delta\subset\mathbb{P}^n$ be the divisor whose strict transform is $D$. We have
\begin{equation}\label{mdeltaodd}
\begin{array}{lll}
\mult_{\sec_k(C)}\Delta & = & 3(h-k+1),\\ 
\mult_{Y_{i,j}^{2k-1}}\Delta & = & 2(h-k+1)+h-k = 3(h-k)+2,\\ 
\mult_{Y_{i,j,r}^{2k}}\Delta & = & 3(h-k),\\ 
\mult_{Y_{i}^{2k}}\Delta & = & h-k+1+2(h-k) = 3h-3k+1.
\end{array} 
\end{equation}
Equalities \ref{mdeltaodd} yield:
$$
\begin{array}{ll}
\pi^{*}(D) = & \widetilde{D}+\sum_{k=1}^h 3(h-k+1)E_{\sec_k(C)}+\sum_{k=1}^{h-1} (3(h-k)+2)\sum_{i,j}E_{Y_{i,j}^{2k-1}} \\ 
 & +\sum_{k=1}^{h-1} (3h-3k+1)\sum_{i}E_{Y_i^{2k}}+\sum_{k=1}^{h-1}3(h-k)E_{Y_{1,2,3}^{2k}},
\end{array} 
$$
and hence
$$
\begin{array}{lll}
K_{\widetilde{X}} = & \pi^{*}(K_{X^{n}_{n+3}}+\epsilon D) & +\sum_{k=1}^h (2h-2k+1-3\epsilon(h-k+1))E_{\sec_k(C)} \\ 
 &  & +\sum_{k=1}^{h-1}(2h-2k+1-\epsilon (3(h-k)+2))\sum_{i,j}E_{Y_{i,j}^{2k-1}}\\
 &  & +\sum_{k=1}^{h-1}(2(h-k)-\epsilon (3h-3k+1))\sum_{i}E_{Y_i^{2k}}\\
 &  & +\sum_{k=1}^{h-1}(2(h-k)-\epsilon (3h-3k))E_{Y_{1,2,3}^{2k}}-\epsilon\widetilde{D}.
\end{array} 
$$
Therefore the pair $(X_{n+3}^n,\epsilon D)$ is klt for any $0\leq \epsilon <\frac{2}{3}$.
\end{proof}


\subsection{The even case $n = 2h$} \label{subset:even}
In this subsection  we construct  divisors $\Delta$ making $X^n_{n+3}$ log Fano
when $n = 2h$ is even.
We follow Notation~\ref{notation:join&secs}.
For each $1\leq i<j \leq n+3$, let $\Delta_{i,j} \subset X_{n+3}^n$ be the strict transform of the divisor $Y^{2h-1}_{i,j}\subset \P^n$,
and  denote by $H_{5,...,n+3} \subset X_{n+3}^n$ the strict transform of
a general hyperplane in $\mathbb{P}^{n+3}$ through $p_5,...,p_{n+3}$.

\begin{Theorem}\label{maineven}
Let $n = 2h\geq 4$ be an even integer. 
Set 
$$
D:=\Delta_{1,2}\cup\Delta_{3,4}\cup \widetilde{\sec_h(C)}\cup H_{5,...,n+3}\subset X_{n+3}^n.
$$
For any $\frac{2h-3}{3h-4}<\epsilon <\frac{2h-1}{3h-1}$ the divisor $-(K_{X_{n+3}^n}+\epsilon D)$ is ample, 
and the pair $(X_{n+3}^n,\epsilon D)$ is klt.
\end{Theorem}

To provide a log resolution of the pair  $(X_{n+3}^n, D)$ in Proposition~\ref{logeven} below, we will need the following result.

\begin{Lemma}\label{intfin}
Any point of $Y^{n-1}_{1,2}\cap Y^{n-1}_{3,4}\subset \P^n$ which is smooth for both divisors $Y^{n-1}_{1,2}$ and $Y^{n-1}_{3,4}$ is a smooth point of 
$Y^{n-1}_{1,2}\cap Y^{n-1}_{3,4}$.
\end{Lemma}

\begin{proof}
Let $x\in \big(Y^{n-1}_{1,2}\cap Y^{n-1}_{3,4}\big)\setminus \big( \Sing(Y^{n-1}_{1,2})\cup \Sing(Y^{n-1}_{3,4})\big)$. 
We shall prove that the intersection of $Y^{n-1}_{1,2}$ and $Y^{n-1}_{3,4}$ is transverse at  $x$, that is, $T_x Y^{n-1}_{1,2}\neq T_x Y^{n-1}_{3,4}$.

Suppose otherwise, and set $P=T_x Y^{n-1}_{1,2}= T_x Y^{n-1}_{3,4}$.
By Terracini's Lemma \cite{Te} we have 
$$
P = \left\langle p_1,p_2,T_{z_1}C,...,T_{z_{h-1}}C\right\rangle = \left\langle p_2,p_3,T_{w_1}C,...,T_{w_{h-1}}C\right\rangle
$$
for suitable $z_i,w_i\in C$, with $z_i\notin \{p_1,p_2\}$ and $w_i\notin \{p_3,p_4\}$.
Set $s = |\{z_1,...,z_{h-1}\}\cap\{w_1,...,w_{h-1}\}|$, $r = |\{z_1,...,z_{h-1}\}\cap\{p_3,p_4\}|$, and $t = |\{w_1,...,w_{h-1}\}\cap\{p_1,p_2\}|$. 
We may assume that $r\geq t$. Then $s\leq h-1-r$, and the number of intersection points in $P\cap C$, counted with multiplicity, 
is at least 
$$
2(2(h-1)-s)+2-r+2-t\geq n+2+r-t\geq n+2.
$$
This is impossible since $C$ has degree $n$.    
\end{proof}

\begin{Proposition}\label{logeven}
Let the assumptions be as in Theorem~\ref{maineven}, and follow Notation~\ref{notation:join&secs}.
For $0\leq m\leq n-3$, we define a modification $X_m$ of  $X_{n+3}^n$ recursively as follows: 
\begin{itemize}
\item[-] $X_0=X_{n+3}^n$,
\item[-] $X_{2k+1}$ is the blow-up of $X_{2k}$ along the strict transforms of $\sec_{k+1}(C)$ and the $Y_{i,j}^{2k+1}$'s
 ($0\leq k\leq h-3$), and also of  $Y_{1,2,3,4}^{2k+1}$ if $0< k\leq h-3$,
\item[-] $X_{2k}$ is the blow-up of $X_{2k-1}$ along the strict transforms of the $Y_i^{2k}$'s, and of the $Y_{i,j,r}^{2k}$'s ($1\leq k\leq h-2$).
\item[-] $X_{n-3}$ is the blow-up of $X_{n-4}$ along the strict transforms of $\sec_{h-1}(C)$, of $Y_{1,2}^{2h-3}$ and of $Y_{3,4}^{2h-3}$. 
\end{itemize}
Then, for any $0\leq m\leq n-4$, the center of the blow-up $X_{m+1}\rightarrow X_m$ is a disjoint union of smooth subvarieties, 
all transverse to the exceptional divisors of $X_m \rightarrow X_0$. 
Moreover, the composition $\pi:X_{n-3}\rightarrow X_{n+3}^n$ of these blow-ups is a log resolution of the pair $(X_{n+3}^n,D)$.
\end{Proposition}

\begin{proof}
Using the same arguments as in the proof of  Proposition~\ref{logodd}, we can prove that, for any $0\leq m\leq n-4$, 
the center of the blow-up $X_{m+1}\rightarrow X_m$ is a disjoint union of smooth subvarieties, 
all transverse to the exceptional divisors of $X_m \rightarrow X_0$. 
Moreover, the strict transforms of $\Delta_{1,2}$, $\Delta_{3,4}$, $\widetilde{\sec_h(C)}$ and  $H_{5,...,2h+3}$ in $X_{n-3}$
are smooth and transverse to the exceptional divisors over $X_0$, and 
the intersection $\widetilde{\sec_h(C)}\cap \widetilde{Y}_{i,j}^{2h-1}$ is transverse.
Clearly the strict transform of  ${H}_{5,...,2h+3}$ is transverse to $\widetilde{\sec_h(C)}$, $\Delta_{1,2}$, $\Delta_{3,4}$  and to all exceptional divisors. 
To show that the strict transform of $D$ in $X_{n-3}$ is simple normal crossing, it remains to compute $\Delta_{1,2}\cap \Delta_{3,4}$.
Note that we cannot use Proposition~\ref{intersections} in this case.
To compute $\Delta_{1,2}\cap \Delta_{3,4}$, we first describe  the intersection of $Y_{i,j}^{n-1}$ and $\Sing(Y_{r,s}^{n-1}) = Y_{r,s}^{n-3}$.

\begin{Claim}\label{claim3}
We have
$$Y_{i,j}^{n-1}\cap Y_{r,s}^{n-3} = Y_{r}^{n-4}\cup Y_{s}^{n-4}\cup Y_{i,r,s}^{n-4}\cup Y_{j,r,s}^{n-4}.$$
Moreover, at a general point in any irreducible component of this intersection, the intersection is transverse.
\end{Claim}

\begin{proof}
Note that $Y_{i,j}^{n-1}\cap Y_{r,s}^{n-3} = (Y_{i,j}^{n-1}\cap Y_{i,j,r,s}^{n-1})\cap Y_{r,s}^{n-3}$. 
By Proposition \ref{intersections},  $Y_{i,j}^{n-1}\cap Y_{i,j,r,s}^{n-1} = Y_{i,j,r}^{n-2}\cup Y_{i,j,s}^{n-2}$. 
Applying Proposition~\ref{intersections} repeatedly, we have that
$$
Y_{i,j,r}^{n-2}\cap Y_{r,s}^{n-3} = (Y_{i,j,r}^{n-2}\cap Y_{i,r,s}^{n-2})\cap Y_{r,s}^{2h-3}=(Y_{i,r}^{n-3}\cup Y_{i,j,r,s}^{n-3})\cap Y_{r,s}^{n-3}
= Y_r^{n-4}\cup Y_{i,r,s}^{n-4}\cup Y_{j,r,s}^{n-4}.
$$ 
Similarly we show that  $Y_{i,j,s}^{n-2}\cap Y_{r,s}^{n-3} = Y_s^{n-4}\cup Y_{i,r,s}^{2h-4}\cup Y_{j,r,s}^{n-4}$.
\end{proof}

The strict transforms $\widetilde{Y}_{1,2}^{n-1}$ and $\widetilde{Y}_{3,4}^{n-1}$ in $X_{n-4}$ are still singular along 
$\widetilde{Y}_{1,2}^{n-3}$ and $\widetilde{Y}_{3,4}^{n-3}$, respectively. However, by Claim \ref{claim3} we have 
$$
\widetilde{Y}_{1,2}^{n-1}\cap \Sing(\widetilde{Y}_{3,4}^{n-1}) = \widetilde{Y}_{3,4}^{n-1}\cap \Sing(\widetilde{Y}_{1,2}^{n-1}) = \emptyset.
$$
Hence, by Lemma \ref{intfin}, in $X_{n-3}$, $\widetilde{Y}_{1,2}^{n-1}\cap \widetilde{Y}_{3,4}^{n-1}$ is smooth, and so 
the intersection $\widetilde{Y}_{1,2}^{n-1}\cap \widetilde{Y}_{3,4}^{n-1}$ is transverse. 
\end{proof}

\begin{proof}[Proof of Theorem~\ref{maineven}]
With  Notation~\ref{notation:generators}, we have 
$$
D = \Delta_{1,2}\cup\Delta_{3,4}\cup \widetilde{\sec_h(C)}\cup H_{5,...,2h+3} \sim  (3h+2)H-(3h-1)(E_1+...+E_{n+3})
$$
and
$$
-K_{X_{n+3}^{n}}-\epsilon D \sim (2h+1-\epsilon (3h+2))H-(2h-1-\epsilon (3h-1))(E_1+...+E_{n+3}).
$$
Recall from Proposition~\ref{mc2} that the Mori cone of $X_{n+3}^n$ is generated by the classes $R_i$'s and $L_{i,j}$'s.
One computes 
$$
(-K_{X_{n+3}^{n}}-\epsilon D)\cdot R_i = 2h-1-\epsilon (3h-1) \ \text{ and } 
(-K_{X_{n+3}^{n}}-\epsilon D)\cdot L_{i,j} = \epsilon (3h-4)-2h+3.
$$
Therefore $-K_{X_{n+3}^n}-\epsilon D$ is ample provided that $\frac{2h-3}{3h-4}<\epsilon <\frac{2h-1}{3h-1}$.

Next we check when the pair $(X_{n+3}^n,\epsilon D)$ is klt.
Let  $\pi:\widetilde{X} :=X_{n-3} \to X^n_{n+3}$ be the log resolution of $(X_{n+3}^n,\epsilon D)$ introduced in Proposition~\ref{logeven} above. 
We have
$$
\begin{array}{ll}
K_{\widetilde{X}} = & \pi^*K_{X_{n+3}^{n}}+\sum_{k=1}^{h-1}(n-2k)E_{\sec_k(C)}+\sum_{k=1}^{h-1}(n-2k)\sum_{i,j}E_{Y_{i,j}^{2k-1}}\\
 & +\sum_{k=1}^{h-2}(n-2k-1)(\sum_{i}E_{Y_{i}^{2k}}+\sum_{i,j,r}E_{Y_{i,j,r}^{2k}})+\sum_{k=2}^{h-2}(n-2k)E_{Y_{1,2,3,4}^{2k-1}}.
\end{array} 
$$
Here we denote by $E_Y$ the exceptional divisor with center $Y\subset \P^n$.

In order to compute discrepancies, we will compute the the multiplicities of 
$\sec_h(C)$, $Y_{1,2}^{2h-1}$, and $Y_{3,4}^{2h-1}$ along the images in $\P^n$ of the subvarieties blown-up by $\pi$.

We start with the divisor $Y_{i,j}^{2h-1}$.
For $1\leq k\leq h-1$, we have:
$$
\mult_{Y_{r,s}^{2k-1}}Y_{i,j}^{2h-1} = \left\lbrace\begin{array}{ll}
\mult_{\sec_{k-1}(C)}\sec_{h-1}(C) = h-k+1 & \rm{if}\: \{\textit{i,j}\}=\{\textit{r,s}\},\\ 
\mult_{\sec_{k}(C)}\sec_{h-1}(C) = h-k & \rm{if}\: \big|\{\textit{i,j}\}\cap\{\textit{r,s}\}\big|=1,\\ 
\mult_{\sec_{k+1}(C)}\sec_{h-1}(C) = h-k-1 & \rm{if}\: \{\textit{i,j}\}\cap\{\textit{r,s}\}=\emptyset.
\end{array}\right. 
$$
$$
\mult_{Y_{r,s,t}^{2k}}Y_{i,j}^{2h-1} = \left\lbrace\begin{array}{ll}
\mult_{\sec_{k+2}(C)}\sec_{h-1}(C) = h-k-2 & \rm{if}\: \textit{i},\{\textit{i,j}\}\cap\{\textit{r,s,t}\}=\emptyset.,\\ 
\mult_{\sec_{k+1}(C)}\sec_{h-1}(C) = h-k-1 & \rm{if}\: \big|\{\textit{i,j}\}\cap\{\textit{r,s,t}\}\big|=1,\\ 
\mult_{\sec_{k}(C)}\sec_{h-1}(C) = h-k & \rm{if}\: \{\textit{i,j}\}\subset \{\textit{r},\textit{s}, \textit{t}\}.
\end{array}\right.
$$
$$
\mult_{Y_{r}^{2k}}Y_{i,j}^{2h-1} = \left\lbrace\begin{array}{ll}
\mult_{\sec_{k+1}(C)}\sec_{h-1}(C) = h-k-1 & \rm{if}\: \textit{r}\notin \{\textit{i,j}\},\\ 
\mult_{\sec_{k}(C)}\sec_{h-1}(C) = h-k & \rm{if}\: \textit{r}\in \{\textit{i,j}\}.
\end{array}\right. 
$$
$$
\mult_{Y_{1,2,3,4}^{2k-1}}Y_{i,j}^{2h-1} =  \mult_{\sec_k(C)}\sec_{h-1}(C) = h-k.
$$

Next we consider the divisor $\sec_h(C)$. For $1\leq k\leq h-1$, we have:
$$
\begin{array}{lll}
\mult_{\sec_k(C)}\sec_h(C) & = & h-k+1,\\
\mult_{Y_{i,j}^{2k-1}}\sec_h(C) & = & \mult_{\sec_{k+1}(C)}\sec_h(C) = h-k,\\ 
\mult_{Y_{i,j,r}^{2k}}\sec_h(C) & = & \mult_{\sec_{k+2}(C)}\sec_h(C) = h-k-1,\\
\mult_{Y_{i}^{2k}}\sec_h(C) & = & \mult_{\sec_{k+1}(C)}\sec_h(C) = h-k,\\ 
\mult_{Y_{1,2,3,4}^{2k-1}}\sec_h(C) & = & \mult_{\sec_{k+2}(C)}\sec_{h}(C) = h-k-1.
\end{array} 
$$

Now let $\bar{D}\subset\mathbb{P}^{n}$ be the divisor whose strict transform is $D$. The above formulas yield:
\begin{equation*}
\begin{array}{lll}
\mult_{\sec_k(C)}\bar{D} & = & 2(h-k)+(h-k+1) = 3h-3k+1,\\ 
\mult_{Y_{i,j}^{2k-1}}\bar{D} & = & 2(h-k)+(h-k) = (h-k+1)+(h-k-1)+(h-k) = 3h-3k,\\ 
\mult_{Y_{i,j,r}^{2k}}\bar{D} & = & (h-k-1)+(h-k)+(h-k-2) = 3h-3k-3,\\ 
\mult_{Y_{i}^{2k}}\bar{D} & = & (h-k-1)+(h-k)+(h-k) =  3h-3k-1,\\
\mult_{Y_{1,2,3,4}^{2k-1}}\bar{D} & = & 2(h-k)+h-k-1 = 3h-3k-1.
\end{array} 
\end{equation*}
Thus 
$$
\begin{array}{ll}
\pi^{*}(D) = & \widetilde{D}+ \sum_{k=1}^{h-1}(3h-3k+1)E_{\sec_k(C)}+\sum_{k=1}^{h-1}(3h-3k)\sum_{i,j}E_{Y_{i,j}^{2k-1}}\\
 & +\sum_{k=1}^{h-2}(3h-3k-1)\sum_{i}E_{Y_{i}^{2k}}+\sum_{k=1}^{h-2}(3h-3k-3)\sum_{i,j,r}E_{Y_{i,j,r}^{2k}}\\
 & +\sum_{k=2}^{h-2}(3h-3k-1)E_{Y_{1,2,3,4}^{2k-1}},
\end{array} 
$$
and hence
$$
\begin{array}{lll}
K_{\widetilde{X}} = & \pi^{*}(K_{X^{n}_{n+3}}+\epsilon D) & +\sum_{k=1}^{h-1} (2h-2k-\epsilon(3h-3k+1))E_{\sec_k(C)} \\ 
 &  & +\sum_{k=1}^{h-1}(2h-2k-\epsilon (3h-3k))\sum_{i,j}E_{Y_{i,j}^{2k-1}}\\
 &  & +\sum_{k=1}^{h-2}(2h-2k-1-\epsilon (3h-3k-1))\sum_{i}E_{Y_{i}^{2k}}\\
  &  & +\sum_{k=1}^{h-2}(2h-2k-1-\epsilon (3h-3k-3))\sum_{i,j,r}E_{Y_{i,j,r}^{2k}}\\
 & & +\sum_{k=2}^{h-2}(2h-2k-\epsilon (3h-3k-1))E_{Y_{1,2,3,4}^{2k-1}}-\epsilon\widetilde{D}.
\end{array} 
$$
For $\epsilon < \frac{2h-1}{3h-2}$ all the discrepancies are greater than $-1$. 
Therefore, for $\frac{2h-3}{3h-4}<\epsilon <\frac{2h-1}{3h-1}$ the divisor $-K_{X_{n+3}^{n}}-\epsilon D$ is ample and the pair $(X_{n+3}^{n},\epsilon D)$ is klt.  
\end{proof}


\section{On a question of Hassett}\label{msc}

In \cite{Ha}, Hassett introduced moduli spaces of weighted pointed curves. 
Given $g\geq 0$ and rational weight data  $A[n] = (a_{1},...,a_{n})$, $0< a_{i}\leq 1$, satisfying $2g-2 + \sum_{i = 1}^{n}a_{i} > 0$,
the moduli space $\overline{M}_{g,A[n]}$ parametrizes  genus $g$ nodal $n$-pointed curves $\{C,(x_1,...,x_n)\}$ 
subject to the following stability conditions:
\begin{itemize}
\item[-] Each $x_i$ is a smooth point of $C$, and the points $x_{i_1}, \dots,  x_{i_k}$ are allowed to coincide only if $\sum_{j= 1}^{k}a_{i_j}\leq 1$.
\item[-] The twisted  dualizing sheaf $\omega_C(a_{1}x_1 +\dots + a_{n}x_n)$ is ample. 
\end{itemize}
In particular,  $\overline{M}_{g,A[n]}$ is a compactification of the moduli space $M_{g,n}$ of genus $g$ smooth $n$-pointed curves. 
The irreducible components of the boundary divisor $\overline{M}_{g,A[n]}\setminus  M_{g,n}$ are well understood. 
In the special case when $g=0$, they are described as follows. 
Consider a partition $I\cup J = \{1,...,n\}$ such that one of the following holds.
\begin{itemize}
\item[-] $I = \{i_1,...,i_r\}$, $J = \{j_1,...,j_{n-r}\}$, with $r,n-r\geq 2$, $a_{i_1}+...+a_{i_r}>1$ and $a_{j_1}+...+a_{j_{n-r}}>1$.
\item[-] $I = \{i_1,i_2\}$ and $i_1+i_2\leq 1$.
\end{itemize}
In the first case, there is a prime divisor $D_{I,J}(A)$ in $\overline{M}_{0,A[n]}$  whose general point corresponds to a nodal curve with two
irreducible components, having marked points $x_{i_1},...,x_{i_r}$ on one component, and $x_{j_1},...,x_{j_{n-r}}$ on the other component.
In the latter case, there is a prime divisor $D_{I,J}(A)$ in $\overline{M}_{0,A[n]}$  parametrizing curves where the marked points $x_{i_1}$ and $x_{i_2}$ coincide. 
These are precisely the boundary divisors of $\overline{M}_{g,A[n]}$.

\begin{say}[{\cite[Section 4]{Ha}}]
For fixed $g$ and $n$, given two collections of rational weight data $A[n]$ and $B[n]$ such that $a_i\geq b_i$ for any $i = 1,...,n$,
there exists a birational \textit{reduction morphism}
$$
\rho_{B[n],A[n]}:\overline{M}_{g,A[n]}\rightarrow\overline{M}_{g,B[n]}.
$$
This morphism associates to a curve $[C,x_1,...,x_n]\in\overline{M}_{g,A[n]}$ the pointed curve obtained by collapsing components of $C$ along which 
$\omega_C(b_1x_1+...+b_nx_n)$ fails to be ample. 
\end{say}

\begin{Example}[{\cite[Sections 6.1 and 6.2]{Ha}}]\label{Example:Hasset}
Consider the weight data
$$
\begin{array}{lll}
A_{0}[n] = (1/(n-2),...,1/(n-2),1),\\ 
A_{1}[n] = (1/(n-3),...,1/(n-3),1),\\ 
A_{1,2}[n] = (1/(n-2),...,1/(n-2), 2/(n-2), 1).
\end{array} 
$$
Then we have $\overline{M}_{0,A_0[n]} \cong \P^{n-3}$, 
$\overline{M}_{0,A_1[n]} \cong X^{n-3}_{n-1} = Bl_{p_1,...,p_{n-1}}\mathbb{P}^{n-3}$ and 
$\overline{M}_{0,A_{1,2}[n]} \cong X^{n-3}_{n-2}=Bl_{p_1,...,p_{n-2}}\mathbb{P}^{n-3}$.
The reduction morphisms $\rho_{A_{1,2}[n],A_1[n]}: X^{n-3}_{n-1}\to X^{n-3}_{n-2}$ and 
$\rho_{A_{0}[n],A_1[n]}: X^{n-3}_{n-1}\to \P^{n-3}$ are the natural blow-up morphisms.  

Let us describe some of the boundary divisors of $\overline{M}_{0,A_1[n]}$
under the blowup morphism $\rho: X^{n-3}_{n-1}\to \P^{n-3}$.
There are $(n-1)$ partitions of  type $I = \{\ihat,n\}$, $J = \{1,...,\ihat,...,n-1\}$. 
The corresponding $(n-1)$ divisors $D_{I,J}$ are the $(n-1)$ exceptional divisors of the blow-up.
There are $\binom{n-1}{2}$ partitions of type $I = \{\ihat_1,\ihat_2\}$, $J = \{1,...,\ihat_1,...,\ihat_2,...,n-1\}\cup \{n\}$.
The corresponding $\binom{n-1}{2}$ divisors $D_{I,J}$ are the strict transforms of the 
$\binom{n-1}{2}$ hyperplanes spanned subsets of cardinality $n-3$ of $\{p_1,...,p_{n-1}\}$. 
\end{Example} 

In \cite{Ha} Hassett proposed the following problem.

\begin{Problem}[{\cite[Problem 7.1]{Ha}}]\label{prob}
Let $A[n]$ be a vector of weights and consider the moduli space $\overline{M}_{0,A[n]}$. Do there exist rational numbers $\alpha_{I,J}$ such that
$$K_{\overline{M}_{0,A[n]}}+\sum_{I,J}\alpha_{I,J}D_{I,J}(A)$$
is ample and the pair $(\overline{M}_{0,A[n]},\sum_{I,J}\alpha_{I,J}D_{I,J}(A))$ is log canonical? 
\end{Problem}

In \cite[Sections 7.1, 7.2, 7.3, Remark 8.5]{Ha} Hassett gives examples in which Problem~\ref{prob} admits a positive answer. 
The techniques developed in this paper allow us to give some more examples. 

\begin{Proposition}\label{phas1}
For the moduli space $\overline{M}_{0,A_1[n]}$, Problem \ref{prob} admits a positive answer.
\end{Proposition}

\begin{proof}
Consider the blow-up  $\rho: \overline{M}_{0,A_1[n]}\cong X^{n-3}_{n-1}\to \P^{n-3}$ described in Example~\ref{Example:Hasset}. 
We denote by $H$ the pullback of the hyperplane class of $\mathbb{P}^{n-3}$. Let $E_1,...,E_{n-1}$ be the exceptional divisors, and
$H_{i_1,...,i_{n-3}}$ be the strict transform of the hyperplane $\left\langle p_{i_1},...,p_{i_{n-3}}\right\rangle$, where $1\leq i_j\leq n-1$. 
Then 
$$
K_{\overline{M}_{0,A_1[n]}} = -(n-2)H+(n-4)(E_1+...+E_{n-1})
$$
and
$$
H_{i_1,...,i_{n-3}} \sim H-E_{i_1}-...-E_{i_{n-3}}.
$$
Recall from Example~\ref{Example:Hasset} that the $E_i$'s and the $H_{i_1,...,i_{n-3}}$'s are boundary divisors of $\overline{M}_{0,A_1[n]}$.
So we set 
$$
\Delta \ = \ \alpha(H_{1,...,n-3}+...+H_{3,...,n-1})+\beta (E_1+...+E_{n-1}),
$$
where $\alpha$ and $\beta$ are positive numbers to be chosen. 
Then 
$$
K_{\overline{M}_{0,A_1[n]}}+\Delta = \left(\alpha\binom{n-1}{2}-n+2\right)H-\left(\alpha\binom{n-2}{2}-n-\beta +4\right)\sum_{i=1}^{n-1}E_i.
$$

Recall from Proposition~\ref{mc2} that the Mori cone of $X^{n-3}_{n-1}\cong \overline{M}_{0,A_1[n]}$ is generated by the classes $R_i$'s and $L_{i,j}$'s
introduced in Section~\ref{moricone}. One computes:
$$
(K_{\overline{M}_{0,A_1[n]}}+\Delta)\cdot R_i = \frac{\alpha}{2}(n-2)(n-3)-n-\beta +4 \ \text { and } \ 
(K_{\overline{M}_{0,A_1[n]}}+\Delta)\cdot L_{i,j} = \frac{\alpha}{2}(n-2)(5-n)+2\beta + n-6.
$$
Therefore $K_{\overline{M}_{0,A_1[n]}}+\Delta$ is ample for $\alpha =\frac{2}{n-2}$ and $\beta = \frac{2}{3}$.

Next we check that the pair $(\overline{M}_{0,A_1[n]},\Delta)$ is log canonical.
Let $\overline{\rho}:Y=\overline{M}_{0,n}\rightarrow\overline{M}_{0,A_1[n]}$ be the composition of  blow-ups introduced in 
Notation~\ref{notation:resolution}. 
It is also a reduction morphism (see \cite[Section 6.1]{Ha}). 
By Proposition \ref{logres}, the morphism $\overline{\rho}$ is a log resolution of the pair $(\overline{M}_{0,A_1[n]},\Delta)$.

There are $\rho_h = \binom{n-1}{h+1}$ $h$-planes spanned by subsets of cardinality $h+1$ of $\{p_1,...,p_{n-1}\}$. 
Each such $h$-plane is contained in $\binom{n-h-2}{n-h-4}$ of the $H_{i_1,...,i_{n-3}}$'s. 
Denote by $E_j^h\subset $, $j = 1,...,\rho_h$,  the exceptional divisors over the $h$-planes. Then we have
$$
K_Y = \overline{\rho}^{*}K_{\overline{M}_{0,A_1[n]}}+\sum_{h=1}^{n-5}(n-h-4)(E_1^h+...+E_{\rho_h}^h)
$$
and 
$$
\overline{\rho}^{*}(\Delta) \sim \sum_{h=1}^{n-5}\alpha\binom{n-h-2}{2}(E_1^h+...+E_{\rho_h}^h)+\alpha\sum_{i_1,...,i_{n-3}}\widetilde{H}_{i_1,...,i_{n-3}}+\beta\sum_i\widetilde{E}_i.
$$
Thus
$$
K_Y+ \widetilde{\Delta} \ = \ \overline{\rho}^*(K_{\overline{M}_{0,A_1[n]}}+\Delta)+\sum_{h=1}^{n-5}\left(n-h-4-\alpha\binom{n-h-2}{2}\right)(E_1^h+...+E^h_{\rho_h}).
$$
For $\alpha = \frac{2}{n-2}$ and $\beta = \frac{2}{3}$ all the discrepancies are greater than $-1$. 
Therefore the pair $(\overline{M}_{0,A_1[n]},\Delta)$ is  log canonical.
\end{proof}

\begin{Proposition}\label{phas2}
For the moduli space $\overline{M}_{0,A_{1,2}[n]}$, Problem \ref{prob} admits a positive answer.
\end{Proposition}

\begin{proof}
Consider the blow-up  $\rho: \overline{M}_{0,A_{1,2}[n]}\cong X^{n-3}_{n-2}\to \P^{n-3}$ described in Example~\ref{Example:Hasset}. 
We denote by $H$ the pullback of the hyperplane class of $\mathbb{P}^{n-3}$.
The prime divisors $D_{I,J}$ appearing in $\Delta$ will be the following:
\begin{itemize}
\item[-] the $(n-2)$ exceptional divisors $E_1,...,E_{n-2}$,
\item[-] the strict transforms $H_{i_1,...,i_{n-3}}$ of the $(n-2)$ hyperplanes spanned by subsets of cardinality $(n-3)$ of $\{p_1,...,p_{n-2}\}$
	($H_{i_1,...,i_{n-3}} \sim H-E_{i_1}-...-E_{i_{n-3}}$),
\item[-] the strict transforms $\Lambda_{j_1,...,j_{n-4}}$ of the $\binom{n-2}{2}$ hyperplanes spanned by subsets of cardinality $(n-4)$ of $\{p_1,...,p_{n-2}\}$ and 
	$p_{n-1}$ ($\Lambda_{j_1,...,j_{n-4}} \sim H-E_{j_1}-...-E_{j_{n-4}}$).
\end{itemize}
Set
$$
\Delta =  \frac{2}{n-2}\sum_{i_1,...,i_{n-3}}H_{i_1,...,i_{n-3}}+\frac{2}{n-2}\sum_{j_1,...,j_{n-4}}\Lambda_{j_1,...,j_{n-4}}+\frac{2}{3}\sum_{i=1}^{n-2}E_i.
$$
Each $p_i$, $i = 1,...,n-2$, lies in exactly $ (n-3)$ of the $H_{i_1,...,i_{n-3}}$'s, and $\binom{n-3}{2}$ of the $\Lambda_{i_1,...,i_{n-3}}$'s. 
So we have
$$
\Delta \sim (n-1)H+\left(\frac{2}{3}-\frac{2(n-3)}{n-2}-\frac{2}{n-2}\binom{n-3}{2}\right)\sum_{i=1}^{n-2}E_i = (n-1)H-\frac{3n-11}{3}\sum_{i=1}^{n-2}E_i
$$
and
$$
K_{\overline{M}_{0,A_{1,2}[n]}}+\Delta = (-n+2+n-1)H+\left(n-4+\frac{11-3n}{3}\right)\sum_{i=1}^{n-2}E_i = H-\frac{1}{3}\sum_{i=1}^{n-2}E_i.
$$

Recall from Proposition~\ref{mc2} that the Mori cone of $X^{n-3}_{n-2}\cong \overline{M}_{0,A_{1,2}[n]}$ is generated by the classes $R_i$'s and $L_{i,j}$'s
introduced in Section~\ref{moricone}. One computes:
$$
(K_{\overline{M}_{0,A_{1,2}[n]}}+\Delta)\cdot R_i = 
(K_{\overline{M}_{0,A_{1,2}[n]}}+\Delta)\cdot L_{i,j} = \frac{1}{3}.
$$
Therefore $K_{\overline{M}_{0,A_{1,2}[n]}}+\Delta$ is ample.

Next we check that the pair $(\overline{M}_{0,A_{1,2}[n]},\Delta)$ is log canonical.
Let $\pi_{n-1}:X_{n-1}^{n-3}\rightarrow X_{n-2}^{n-3}$ be the blow-up of $p_{n-1}$ and consider the composition 
  \[
  \begin{tikzpicture}[xscale=4.5,yscale=-2.2]
    \node (A0_0) at (0, 0) {$Y$};
    \node (A0_1) at (1, 0) {$X^{n-3}_{n-1} = \overline{M}_{0,A_1[n]}$};
    \node (A0_2) at (2, 0) {$X^{n-3}_{n-2} = \overline{M}_{0,A_{1,2}[n]}$,};
    \path (A0_0) edge [->]node [auto] {$\scriptstyle{\overline{\rho}}$} (A0_1);
    \path (A0_1) edge [->]node [auto] {$\scriptstyle{\pi_{n-1}}$} (A0_2);
    \path (A0_0) edge [->,swap,bend right=-25]node [auto] {$\scriptstyle{\widetilde{\rho}}$} (A0_2);
  \end{tikzpicture}
  \]
where $\overline{\rho}$ is the log resolution used in the proof of Proposition \ref{phas1}. 
Then $\widetilde{\rho}$ is a log resolution of the pair $(\overline{M}_{0,A_{1,2}[n]},D)$. 
Let $E_{n-1}$ be the exceptional divisor over $p_{n-1}$.
There are $\gamma_h = \binom{n-2}{h+1}$ $h$-planes spanned by subsets of cardinality $h+1$ of $\{p_1,...,p_{n-2}\}$.
We denote by $E_{j}^h$, $1\leq j\leq \gamma_h$, the exceptional divisors over these $h$-planes.
Similarly, there are $\overline{\gamma}_h = \binom{n-2}{h}$  $h$-planes spanned by $p_{n-1}$ and subsets of cardinality $h$ of $\{p_1,...,p_{n-2}\}$.
We denote by  $\overline{E}_{j}^h$, $ 1\leq j\leq \overline{\gamma}_h$, the  exceptional divisors over  these $h$-planes.
Note that
\begin{itemize}
\item[-] the point $p_{n-1}$ is contained all of the $\binom{n-2}{2}$ $\Lambda_{j_1,...,j_{n-4}}$'s, 
\item[-] any $h$-plane spanned by subsets of cardinality $h+1$ of $\{p_1,...,p_{n-2}\}$ is contained in $n-h-3$ of the $H_{i_1,...,i_{n-3}}$'s and in 
$\binom{n-h-3}{2}$ of the $\Lambda_{j_1,...,j_{n-4}}$'s,
\item[-] any $h$-plane spanned by $p_{n-1}$ and subsets of cardinality $h$ of $\{p_1,...,p_{n-2}\}$  is contained in $\binom{n-h-2}{2}$ of the 
$\Lambda_{j_1,...,j_{n-4}}$'s.
\end{itemize}
Therefore, we have
$$
\begin{array}{ll}
\widetilde{\rho}^*(\Delta) = & \frac{2}{n-2}\binom{n-2}{2}E_{n-1}+\frac{2}{n-2}\sum_{h=1}^{n-5}\left(n-h-3+\binom{n-h-3}{2}\right)(E_1^h+...+E_{\gamma_h}^h)+ \\ 
 & \frac{2}{n-2}\sum_{h=1}^{n-5}\binom{n-h-2}{2}(\overline{E}_1^h+...+\overline{E}_{\overline{\gamma}_h}^h)+\widetilde{\Delta}.
\end{array} 
$$
Since 
$$
K_{Y} = \widetilde{\rho}^*K_{\overline{M}_{0,A_{1,2}[n]}}+(n-4)E_{n-1}+\sum_{h=1}^{n-5}(n-h-4)(E_1^h+...+E_{\gamma_h}^h+
\overline{E}_1^h+...+\overline{E}_{\gamma_h}^h),
$$
we have
$$
\begin{array}{ll}
K_{Y} +  \widetilde{\Delta} = & \widetilde{\rho}^*(K_{\overline{M}_{0,A_{1,2}[n]}}+\Delta)+\left(n-4-\frac{2}{n-2}\binom{n-2}{2}\right)E_{n-1}+\\
 & \sum_{h=1}^{n-5}\left(n-h-4-\frac{2}{n-2}\left(n-h-3+\binom{n-h-3}{2}\right)\right)(E_1^h+...+E_{\gamma_h}^h)+\\
 & \sum_{h=1}^{n-5}\left(n-h-4-\frac{2}{n-2}\binom{n-h-2}{2}\right)(\overline{E}_1^h+...+\overline{E}_{\overline{\gamma}_h}^h).
\end{array} 
$$
The discrepancies are all $\geq -1$, and  hence the pair $(\overline{M}_{0,A_{1,2}[n]},\Delta)$ is log canonical.
\end{proof}

\begin{Remark}
Any $3$-dimensional Hassett's space $\overline{M}_{0,A[6]}$ admits a  reduction morphism $\rho:\overline{M}_{0,6}\rightarrow\overline{M}_{0,A[6]}$
(\cite[Theorem 4.1]{Ha}).
The moduli space $\overline{M}_{0,6}$ is log Fano by \cite{HK}.
So, by \cite[Corollary 1.3]{GOST}, $\overline{M}_{0,A[6]}$ is also log Fano.
Examples of  $3$-dimensional Hassett's spaces are the following.
\begin{itemize}
\item[-] The blow-up of $\mathbb{P}^3$ in four general points, along the strict transforms of the lines spanned by them, and in a fifth general point. 
	This variety corresponds to $A[6] = (1/3,1/3,1/3,1/3,1,1)$. 
\item[-] The blow-up of $\mathbb{P}^3$ in five general points, and along the strict transforms of the lines spanned by them.This is $\overline{M}_{0,6}$ itself.
\item[-] The blow-up $X_1$ of $\mathbb{P}^1_1\times\mathbb{P}^1_2\times\mathbb{P}^1_3$ in $p_1 = ([0:1],[0:1],[0:1])$, $p_2 = ([1:0],[1:0],[1:0])$, 
	and $p_3 = ([1:1],[1:1],[1:1])$. 
	This variety corresponds to $A_1[6] = (2/3,2/3,2/3,1/6,1/6,1/6)$ (see  \cite[Section 6.3]{Ha}).
\item[-] Consider the projections $\pi_i:\mathbb{P}^1_1\times\mathbb{P}^1_2\times\mathbb{P}^1_3\rightarrow\mathbb{P}^1_i$, and set
	 $F_0 = \bigcup_{i=1}^3\pi_i^{-1}([0:1])$, $F_1 = \bigcup_{i=1}^3\pi_i^{-1}([1:0])$, $F_{\infty} = \bigcup_{i=1}^3\pi_i^{-1}([1:1])$. 
	 Let $\Delta_2$ be the union of the $2$-dimensional diagonals of $\mathbb{P}^1_1\times\mathbb{P}^1_2\times\mathbb{P}^1_3$. 
	 Let $X_2$ be the blow-up of $X_1$ along the strict transform of $\Delta_2\cap (F_0\cup F_1\cup F_{\infty})$. 
	 This variety corresponds to $A_{2}[6] = (2/3,2/3,2/3,1/3,1/3,1/3)$ (see  \cite[Section 6.3]{Ha}).
\item[-] The blow-up $X_3$ of $X_2$ along the strict transform of the $1$-dimension diagonal $\Delta_1$ of 
	$\mathbb{P}^1_1\times\mathbb{P}^1_2\times\mathbb{P}^1_3$. This is $\overline{M}_{0,6}$ (see \cite[Section 6.3]{Ha}).
\end{itemize}
\end{Remark}

Let  $q_1 =([1:0],...,[1:0])$, $q_2 = ([0:1],...,[0:1])$, $q_3 = ([1:1],...,[1:1])\in (\mathbb{P}^1)^{n-3}$,
and set $Y_3^{n-3} = Bl_{q_1,q_2,q_3}(\mathbb{P}^1)^{n-3}$.
By \cite[Section 6.3]{Ha}, $Y_3^{n-3}\cong \overline{M}_{0,A[n]}$ for $A[n] = (\frac{2}{3},\frac{2}{3},\frac{2}{3},\frac{1}{3(n-4)},...,\frac{1}{3(n-4)})$.

\begin{Proposition}
Then there exists a small birational modification
$$
X_{n-1}^{n-3}\dasharrow Y_3^{n-3} .
$$
In particular, $Y_{3}^{n-3}$ is log Fano.
\end{Proposition}

\begin{proof}
Note that the Picard numbers satisfy $\rho(X_{n-1}^{n-3}) = \rho(Y_3^{n-3}) = n$. 
Without lost of generality, we may assume that $X_{n-1}^{n-3}$ is the blow-up of $\P^{n-3}$ at the points
 $p_1 = [1:0:...:0]$, $p_2 = [0:1:...:0]$,..., $p_{n-2}=[0:...:0:1]$ and $p_{n-1}=[1:1:...:1]$. 
Set $X_{n-2}^{n-3} = Bl_{p_1,...,p_{n-2}}\mathbb{P}^{n-3}$, $X_{n-3}^{n-3} = Bl_{p_1,...,p_{n-3}}\mathbb{P}^{n-3}$, 
$Y_{2}^{n-3} = Bl_{q_1,q_2}(\mathbb{P}^1)^{n-3}$ and  and $Y_{1}^{n-3} = Bl_{q_1}(\mathbb{P}^1)^{n-3}$. 
These are all toric varieties. Let $e_1,...,e_{n-3}$ be the standard basis vectors of the co-character lattice of $(k^{*})^{n-3}$. 
The rays of the fan of $\mathbb{P}^{n-3}$ are $e_1,...,e_{n-3}$ and $-e_1-...-e_{n-3}$. 
By blowing-up $p_1,...,p_{n-2}$ we add the rays $-e_1,...,-e_{n-3}$ and $e_1+...+e_{n-3}$. 
On the other hand the rays of $(\mathbb{P}^1)^{n-3}$ are $e_1,...,e_{n-3},-e_1,...,-e_{n-3}$, 
and blowing-up  $q_1,q_2$ corresponds to introducing the two rays $e_1+...+e_{n-3}$ and $-e_{1}-...-e_{n-3}$. 
So the fans of $X_{n-2}^{n-3}$ and $Y_{2}^{n-3}$ have the same $1$-dimensional rays. 
Therefore, $X_{n-2}^{n-3}$ and $Y_{2}^{n-3}$  are isomorphic in codimension one.

Given $1\leq i_1< ... <i_{n-4}\leq n-3$, set $H_{i_1,...,i_{n-4}}^{n-5} = \left\langle p_{i_1},...,p_{i_{n-4}}\right\rangle$, and
$\{j_1,j_2\} = \{0,...,n-3\}\setminus\{i_1-1,...,i_{n-4}-1\}$. The projection from $H_{i_1,...,i_{n-4}}^{n-5}$ is the rational map
$$ 
\begin{array}{lccc}
  \pi_{i_1,...,i_{n-4}}: & \mathbb{P}^{n-3} & \dasharrow & \mathbb{P}^1 \\ 
   & \left[x_0:...:x_{n-3}\right] & \mapsto & [x_{j_1}:x_{j_2}].
\end{array}   
$$
There are $(n-3)$ of those, inducing a rational map 
$$ 
\begin{array}{lccc}
  g: & \mathbb{P}^{n-3} & \dasharrow & (\mathbb{P}^1)^{n-3} \\ 
   & x=\left[x_0:...:x_{n-3}\right] & \mapsto & (\pi_{1,...,n-4}(x),...,\pi_{2,...,n-3}(x)).
\end{array}   
$$
The hyperplane $W = \left\langle p_1,...,p_{n-3}\right\rangle = \{x_{n-3}=0\}$ is mapped to the point $q_1\in  (\mathbb{P}^1)^{n-3}$ by $g$.
This is the only divisor contracted by $g$. 
Therefore, by blowing-up $q_1\in (\mathbb{P}^1)^{n-3}$ we obtain a small transformation 
$g_1:X_{n-3}^{n-3}\dasharrow Y_1^{n-3} $ fitting in the following diagram:
  \[
  \begin{tikzpicture}[xscale=2.5,yscale=-1.2]
    \node (A0_0) at (0, 0) {$X_{n-3}^{n-3}$};
    \node (A0_1) at (1, 0) {$Y_1^{n-3}$};
    \node (A1_0) at (0, 1) {$\mathbb{P}^{n-3}$};
    \node (A1_1) at (1, 1) {$(\mathbb{P}^1)^{n-3}$};
    \path (A0_0) edge [->,dashed]node [auto] {$\scriptstyle{g_1}$} (A0_1);
    \path (A1_0) edge [->,dashed]node [auto] {$\scriptstyle{g}$} (A1_1);
    \path (A0_1) edge [->,]node [auto] {$\scriptstyle{\psi_{1}}$} (A1_1);
    \path (A0_0) edge [->,swap]node [auto] {$\scriptstyle{\phi_{n-3}}$.} (A1_0);
  \end{tikzpicture}
  \]

Note that $g_1$ maps the strict transform $\widetilde{W}$ of $W$ to the exceptional divisor $E_{q_1}$, while the exceptional divisors 
$E_{p_1},...,E_{p_{n-3}}$ are mapped to the strict transforms of the $(n-3)$ divisors in $(\mathbb{P}^1)^{n-3}$ obtained by fixing one of the factors.
Note also that  $g([0:...:0:1]) = ([0:1],...,[0:1])$ and $g([1:...:1]) = ([1:1],...,[1:1])$. 
It follows from  the universal property of the blow-up that $g_1$ lifts to a small modification $f:X_{n-1}^{n-3}\dasharrow Y_3^{n-3}$ mapping $E_{p_{n-2}}$ to $E_{q_{2}}$, and $E_{p_{n-3}}$ to $E_{q_3}$.
\end{proof}


\end{document}